\documentclass{article}
\usepackage[latin1]{inputenc}
\usepackage[T1]{fontenc}
\usepackage{amsmath,amssymb}
\usepackage{amsfonts,amsthm}

\usepackage{geometry}
\usepackage{color}
   \definecolor{cites}{rgb}{0.75 , 0.00 , 0.00}  
   \definecolor{urls} {rgb}{0.00 , 0.00 , 1.00}  
   \definecolor{links}{rgb}{0.00 , 0.00 , 0.5}   
\usepackage{graphicx}
\usepackage{cancel}
\usepackage[plainpages=false,colorlinks=true, citecolor=cites, urlcolor=urls, linkcolor=links, breaklinks=true]{hyperref}

\newcommand{\B}{\mathbb{B}}
\newcommand{\C}{\mathbb{C}}
\newcommand{\N}{\mathbb{N}}
\newcommand{\R}{\mathbb{R}}
\newcommand{\Z}{\mathbb{Z}}

\newcommand{\Kc}{\mathcal{K}}
\newcommand{\Lc}{\mathcal{L}}
\newcommand{\Sc}{\mathcal{S}}

\newcommand{\Mf}{\mathfrak{M}}
\newcommand{\Tf}{\mathfrak{T}}

\newcommand{\BDO}{\textup{BDO}}

\renewcommand{\epsilon}{\varepsilon}

\newcommand{\vertiii}[1]{{\left\vert\kern-0.25ex\left\vert\kern-0.25ex\left\vert #1
    \right\vert\kern-0.25ex\right\vert\kern-0.25ex\right\vert}}
\newcommand{\vertiiis}[1]{{\vert\kern-0.25ex\vert\kern-0.25ex\vert #1
    \vert\kern-0.25ex\vert\kern-0.25ex\vert}}

\DeclareMathOperator{\coker}{coker}
\DeclareMathOperator{\diam}{diam}
\DeclareMathOperator{\dist}{dist}
\DeclareMathOperator{\im}{im}
\DeclareMathOperator{\ess}{ess}
\DeclareMathOperator*{\slim}{s-\lim}
\DeclareMathOperator{\spec}{sp}
\DeclareMathOperator{\supp}{supp}

\DeclareMathOperator{\BUC}{BUC}

\DeclareMathOperator{\Osc}{Osc}
\DeclareMathOperator{\VO}{VO}
\DeclareMathOperator{\VMO}{VMO}
\DeclareMathOperator{\BC}{BC}

\newcommand{\from}{\colon}

\providecommand{\scpr}[2]{\left\langle #1, #2 \right\rangle}
\renewcommand{\sp}{\scpr}
\providecommand{\abs}[1]{\left\lvert#1\right\rvert}
\providecommand{\norm}[1]{\left\lVert#1\right\rVert}
\providecommand{\set}[1]{\left\{ #1\right\}}

\newtheorem{thm}{Theorem}
\newtheorem{mthm}{Theorem}

\newtheorem{lem}[thm]{Lemma}
\newtheorem{prop}[thm]{Proposition}
\newtheorem{cor}[thm]{Corollary}

\theoremstyle{definition}
\newtheorem{defn}[thm]{Definition}
\newtheorem{rem}[thm]{Remark}

\numberwithin{equation}{section}

\allowdisplaybreaks
\begin{document}
\title{\bf The Essential Spectrum of Toeplitz Operators on the Unit Ball}
\author{Raffael Hagger}
\maketitle
\vspace{-0.4cm}
\begin{abstract}
In this paper we study the Fredholm properties of Toeplitz operators acting on weighted Bergman spaces $A^p_{\nu}(\B^n)$, where $p \in (1,\infty)$ and $\B^n \subset \C^n$ denotes the $n$-dimensional open unit ball. Let $f$ be a continuous function on the Euclidean closure of $\B^n$. It is well-known that then the corresponding Toeplitz operator $T_f$ is Fredholm if and only if $f$ has no zeros on the boundary $\partial\B^n$. As a consequence, the essential spectrum of $T_f$ is given by the boundary values of $f$. We extend this result to all operators in the algebra generated by Toeplitz operators with bounded symbol (in a sense to be made precise down below). The main ideas are based on the work of Su\'arez et al.~(\cite{MiSuWi,Suarez}) and limit operator techniques coming from similar problems on the sequence space $\ell^p(\Z)$ (\cite{HaLiSe,LiSe,RaRoSi} and references therein).


\end{abstract}

\section{Introduction} \label{Introduction}

Consider some measure space $(X,\mu)$ and a corresponding $L^p$-space for some $p \in (1,\infty)$, say. Further assume that there is a bounded projection $P$ onto a closed subspace $\Sc$ of $L^p(X,\mu)$. If we now decompose a multiplication operator parallel to this projection, we obtain a Toeplitz operator. More precisely, if $f \from X \to \C$ is an essentially bounded function and $M_f$ the corresponding multiplication operator on $L^p(X,\mu)$, the corresponding Toeplitz operator is given by $T_f := PM_f|_{\Sc}$. Toeplitz operators are one of the prime examples for non-normal operators and are thus extensively studied on various different domains, the most prominent example probably being the Hardy space over the circle. In our case here we are going to consider Toeplitz operators on the space of holomorphic $L^p$-functions defined on the complex open unit ball $\B^n$, which is called a Bergman space. Using the variables above, we consider $X = \B^n$ with a weighted Lebesgue measure $\mathrm{d}v_{\nu}$ for a weight parameter $\nu$ and we take $\Sc$ to be the closed subspace of holomorphic functions contained in $L^p_{\nu} := L^p(\B^n,\mathrm{d}v_{\nu})$, here denoted by $A^p_{\nu}$ (see Section \ref{Notation} for more details).

In this paper we are particularly interested in the Fredholm properties of Toeplitz operators. Recall that in the Hardy space case, a Toeplitz operator $T_f$ with a continuous symbol $f$ is Fredholm if and only if $f$ does not have any zeros. A similar result holds for the Bergman space: A Toeplitz operator $T_f$ with a symbol that can be continuously extended to the (Euclidean) boundary $\partial\B^n$ is Fredholm if and only if $f$ has no zeros on the boundary. This result was first established by Coburn in \cite{Coburn} and then generalized in many different directions by several authors (e.g.~\cite{ACM,BaCo,BeCo,ChoLe,McDonald,MiSuWi,StroeZhe,Suarez,Zeng,Zhu87}). One of the latest improvements (\cite[Theorem 10.3]{Suarez}, \cite[Theorem 5.8]{MiSuWi}) include the following result: Let $\Tf_{2,\nu} \subset \Lc(A^2_{\nu})$ denote the closed subalgebra generated by Toeplitz operators $T_f$ with $f \in L^{\infty}(\B^n)$. Then $A \in \Tf_{2,\nu}$ is Fredholm if and only if all of its limit operators are invertible and their inverses are uniformly bounded. Roughly speaking, limit operators are operators that appear when we shift our operator $A$ to the boundary of the domain (a more accurate definition is given in Section \ref{Limit_operators}). This theorem reminds of a seemingly unrelated result in the Fredholm theory of sequence spaces $\ell^p$. There, until a few years ago, one of the main theorems was stated as follows: A band-dominated\footnote{a certain property related to the structure of the corresponding infinite matrix} operator $A$ is Fredholm if and only if all of its limit operators are invertible and their inverses are uniformly bounded (see e.g.~\cite{RaRoSi}). There are a few problems with this characterization. Not only is the uniform boundedness condition difficult to work with, it also prevents us from writing the essential spectrum as the union of spectra of limit operators. As a consequence, many different authors worked out particular examples (see e.g.~\cite[Chapter 3]{Lindner} for a summary) where the uniform boundedness condition could be dropped. Moreover, as there was no known example where the uniform boundedness condition was actually violated, it was conjectured that this condition was actually redundant. And indeed, this was shown in \cite{LiSe} a few years ago. Now the goal of this paper is to show that the same is the case for Toeplitz operators on the Bergman space:

\begin{mthm} \label{main_result}
Let $\Tf_{p,\nu} \subset \Lc(A^p_{\nu})$ denote the closed subalgebra generated by Toeplitz operators with bounded symbol and $A \in \Tf_{p,\nu}$. Then the following are equivalent:
\begin{itemize}
\item[$(i)$] $A$ is Fredholm,
\item[$(ii)$] $A_x$ is invertible for all $x \in \Mf \setminus \B^n$ and $\sup\limits_{x \in \Mf \setminus \B^n} \norm{A_x^{-1}} < \infty$,
\item[$(iii)$] $A_x$ is invertible for all $x \in \Mf \setminus \B^n$.
\end{itemize}
\end{mthm}

Here, the $A_x$ denote the limit operators of $A$ and they are indexed over the boundary of a certain compactification $\Mf$ of $\B^n$. In particular, we extend \cite[Theorem 10.3]{Suarez} and \cite[Theorem 5.8]{MiSuWi} to the Banach space case $p \neq 2$ and show that the uniform boundedness condition is redundant just like it is in the sequence space case. As a consequence, we get
\[\spec_{\ess}(A) = \bigcup\limits_{x \in \Mf \setminus \B^n} \spec(A_x)\]
for all operators $A \in \Tf_{p,\nu}$.

The paper is organized as follows. In Section \ref{Notation} we introduce all the necessary notation and some preliminary results. Then we proceed by introducing what, in analogy to the sequence space case, we will call band-dominated operators and show some basic properties in Section \ref{BDO}. In particular, we show that Toeplitz operators are band-dominated. In Section \ref{Fredholm_criterion} we show a Fredholm criterion for band-dominated operators that will be crucial for the proof of Theorem \ref{main_result}. In Section \ref{Limit_operators} we introduce limit operators and finally show our main theorem. After that, we proceed by showing that a similar result holds for the essential norm of an operator $A \in \Tf_{p,\nu}$ in Section \ref{Norm_estimates}. However, our result is less complete in this case (compare with the corresponding result on $\ell^p$: \cite[Theorem 3.2]{HaLiSe}) and leaves some questions open. Section \ref{VO} is devoted to some applications of Theorem \ref{main_result}.

Note that similar results are expected to hold for more general domains. Sections \ref{Notation} to \ref{Fredholm_criterion} are in fact valid word by word for any bounded symmetric domain $\Omega \subset \C^n$. For future reference we therefore chose to provide full generality in these sections. However, in Sections \ref{Limit_operators} and \ref{Norm_estimates} there are some open problems in the most general case, which is part of the reason why we restrict ourselves to $\Omega = \B^n$ in this paper. A more general setting will be the topic of future work. The reader who is only interested in the unit ball (which is quite frankly the topic of this paper) may replace any $\Omega$ by $\B^n$, ignore the specific notions for bounded symmetric domains and use the respective (more explicit) formulas.

\section{Notation and preliminary results} \label{Notation}

Let $\Omega \subset \C^n$ be an irreducible bounded symmetric domain of genus $g$ and type $(r,a,b)$ in its Harish-Chandra realization with corresponding Bergman metric $\beta$ and Jordan triple determinant $h$. As we only need a handful of properties of bounded symmetric domains, we do not provide an introduction to these notions here. Instead, we refer to \cite{Englis,Upmeier} for introductions and just mention the properties we actually need. For $p \in (1,\infty)$ and $\nu > -1$ denote by $L^p_{\nu}$ the usual Lebesgue space of $p$-integrable functions on $(\Omega,\mathrm{d}v_{\nu})$, where
\[\mathrm{d}v_{\nu}(z) = c_{\nu}h(z,z)^{\nu} \, \mathrm{d}v(z),\]
$\mathrm{d}v$ is the usual Lebesgue measure restricted to $\Omega$ and $c_{\nu}$ is a normalizing constant chosen such that $\mathrm{d}v_{\nu}(\Omega) = 1$. The (unique) geodesic symmetry interchanging $0$ and $z$ is denoted by $\phi_z$. In particular, these symmetries $\phi_z$ satisfy
\[\beta(\phi_z(x),\phi_z(y)) = \beta(x,y)\]
for all $x,y \in \Omega$. Moreover, $h$ and $\mathrm{d}v_{\nu}$ transform under $\phi_z$ as follows:
\begin{align*}
h(\phi_z(x),\phi_z(y)) &= \frac{h(z,z)h(x,y)}{h(x,z)h(z,y)}\\
\mathrm{d}v_{\nu}(\phi_z(w)) &= \frac{h(z,z)^{\nu+g}}{\abs{h(w,z)}^{2(\nu+g)}} \mathrm{d}v_{\nu}(w)
\end{align*}
(see \cite{Englis} for details).

Note that for the unit ball $\Omega = \B^n$ we have $(r,a,b) = (1,2,n-1)$ and $g = n+1$. Moreover, the Bergman metric is the usual hyperbolic metric on the unit ball and the Jordan triple determinant is simply given by $h(z,w) = 1 - \sp{z}{w}$ in this case. In case $n = 1$, $\phi_z$ is given explicitly by the M\"obius transform $w \mapsto \frac{z-w}{1-w\bar{z}}$. We refer to \cite{Zhu} for an explicit description of $\phi_z$ in higher dimensions.

The (closed) subspace of holomorphic functions contained in $L^p_{\nu}$ is denoted by $A^p_{\nu}$ and called a weighted Bergman space. The set of bounded linear operators between Banach spaces $X$ and $Y$ is denoted by $\Lc(X,Y)$ and we abbreviate $\Lc(X) := \Lc(X,X)$. The set of compact operators in $\Lc(X)$ will be denoted by $\Kc(X)$. $A \in \Lc(X)$ is called Fredholm if it is invertible modulo $\Kc(X)$, i.e.~if there exists $B \in \Lc(X)$ such that both $AB-I$ and $BA-I$ are compact. Equivalently, $A$ is Fredholm if and only if $\ker A$ and $\coker A$ are both finite-dimensional (Atkinson's theorem). The essential spectrum of an operator $A$ will be denoted by $\spec_{\ess}(A)$ and is given by
\[\spec_{\ess}(A) = \set{\lambda \in \C : A - \lambda I \text{ is not Fredholm}}.\]
We will say that a net (sequence, series, etc.) of operators converges $*$-strongly if the net converges strongly and the net of adjoints converges strongly to the adjoint. The characteristic function of a set $M$ will be denoted by $\chi_M$.

Let $P_{\nu} \in \Lc(L^2_{\nu})$ be the orthogonal projection onto $A^2_{\nu}$, called the Bergman projection. One can show (see e.g.~\cite{Englis} or \cite{Upmeier}) that $P_{\nu}$ is given by
\[(P_{\nu}f)(z) = \int_{\Omega} f(w)h(z,w)^{-\nu-g} \, \mathrm{d}v_{\nu}(w).\]
Using the same formula also for $p \neq 2$, we can define for every $f \in L^{\infty}(\Omega)$ the corresponding Toeplitz operator $T_f := P_{\nu}M_f|_{A^p_{\nu}}$, where $(M_f)g = f \cdot g$ for all $g \in L^p_{\nu}$. $T_f$ then defines a bounded linear operator on $A^p_{\nu}$ with $\norm{T_f} \leq \norm{P_{\nu}}\norm{f}_{\infty}$, provided that $P_{\nu}$ is indeeed a bounded linear operator on $L^p_{\nu}$. The function $f$ is called the symbol of $T_f$ and $M_f$, respectively. The algebra generated by all Toeplitz operators acting on $A^p_{\nu}$ will be denoted by $\Tf_{p,\nu}$.

The next proposition provides a sufficient condition for $P_{\nu}$ to be bounded. Note that this is certainly not optimal if $r > 1$ as the case $\alpha = \nu$, $p = 2$ demonstrates. For a more optimal condition in the case $\alpha = \nu$ we refer to \cite[Lemma 9]{Englis} (which is a special case of \cite[Theorem II.7]{BeTe}).

\begin{prop} \label{BoundedProjections}
Let $p(\alpha+1) > \nu + 1 + \frac{(r-1)a}{2} > p\frac{(r-1)a}{2}$. Then $P_{\alpha} \in \Lc(L^p_{\nu})$ and $P_{\alpha}f = f$ for all $f \in A^p_{\nu}$. In particular, $P_{\alpha}$ is a bounded projection onto $\im(P_{\alpha}) = A^p_{\nu}$.
\end{prop}

\begin{proof}
By definition,
\[\abs{(P_{\alpha}f)(z)} \leq \int_{\Omega} \abs{h(z,w)}^{-\alpha-g}\abs{f(w)} \, \mathrm{d}v_{\alpha}(w) = \frac{c_{\alpha}}{c_{\nu}}\int_{\Omega} \abs{h(z,w)}^{-\alpha-g}h(w,w)^{\alpha-\nu}\abs{f(w)} \, \mathrm{d}v_{\nu}(w).\]
We want to show that the integral operator with kernel $R(z,w) = \abs{h(z,w)}^{-\alpha-g}h(w,w)^{\alpha-\nu}$ is bounded on $L^p_{\nu}$. To do this, we apply the Schur test with the test function $h(z) := h(z,z)^s$, where $s \in \R$ is to be determined later. We thus need to show that there exists a constant $C$ such that
\begin{align*}
\int_{\Omega} \abs{h(z,w)}^{-\alpha-g}h(w,w)^{\alpha-\nu}h(w,w)^{sq} \, \mathrm{d}v_{\nu}(w) &= c_{\nu}\int_{\Omega} \abs{h(z,w)}^{-\alpha-g}h(w,w)^{sq+\alpha} \, \mathrm{d}v(w)\\
&\leq Ch(z,z)^{sq},\\
\int_{\Omega} \abs{h(z,w)}^{-\alpha-g}h(w,w)^{\alpha-\nu}h(z,z)^{sp} \, \mathrm{d}v_{\nu}(z) &= c_{\nu}h(w,w)^{\alpha-\nu}\int_{\Omega} \abs{h(z,w)}^{-\alpha-g}h(z,z)^{sp+\nu} \, \mathrm{d}v(z)\\
&\leq Ch(w,w)^{sp},
\end{align*}
where $\frac{1}{p} + \frac{1}{q} = 1$. By \cite[Theorem 4.1]{FaKo}, the first inequality holds for $s \in (-\frac{\alpha+1}{q},-\frac{(r-1)a}{2q})$, whereas the second one holds for $s \in (-\frac{\nu+1}{p},-\frac{\nu}{p} + \frac{\alpha}{p} - \frac{(r-1)a}{2p})$. A simple computation shows that $(-\frac{\alpha+1}{q},-\frac{(r-1)a}{2q}) \cap (-\frac{\nu+1}{p},-\frac{\nu}{p} + \frac{\alpha}{p} - \frac{(r-1)a}{2p})$ is non-empty if the inequalities stated in the proposition are assumed (cf.~\cite[Theorem 2.11]{Zhu} in the case $\Omega = \B^n$). Thus $P_{\alpha} \in \Lc(L^p_{\nu})$.

As $P_{\alpha} \from L^2_{\alpha} \to A^2_{\alpha}$ is the orthogonal projection, we get $P_{\alpha}f = f$ for $f \in A^2_{\alpha} \cap A^p_{\nu}$. As $A^2_{\alpha} \cap A^p_{\nu}$ is dense in $A^p_{\nu}$ (polynomials are dense), this generalizes to all of $A^p_{\nu}$. Similarly, $P_{\alpha}f \in A^p_{\nu}$ for $f \in L^2_{\alpha} \cap L^p_{\nu}$ generalizes to all of $L^p_{\nu}$. Therefore $P_{\alpha} \from L^p_{\nu} \to A^p_{\nu}$ is a bounded projection as well.
\end{proof}

For the unit ball $\B^n$ the condition on $\nu$ and $\alpha$ simplifies significantly because we have $r = 1$ in this case.

\begin{cor} \label{BoundedProjectionsCor}
Let $\Omega = \B^n$ and $p(\alpha+1) > \nu + 1 > 0$. Then $P_{\alpha} \in \Lc(L^p_{\nu})$ and $P_{\alpha}f = f$ for all $f \in A^p_{\nu}$. In particular, $P_{\alpha}$ is a bounded projection onto $\im(P_{\alpha}) = A^p_{\nu}$.
\end{cor}

\begin{defn}
We will call $(\alpha,\nu,p) \in \R^2 \times (1,\infty)$ an admissible triple (for $\Omega$) if the inequalities in Proposition \ref{BoundedProjections} are satisfied.
\end{defn}

Note that by Corollary \ref{BoundedProjectionsCor}, $(\nu,\nu,p)$ is always admissible if $\Omega = \B^n$. Therefore the Toeplitz algebra $\Tf_{p,\nu}$ can be defined for all $\nu > -1$ and $p \in (1,\infty)$ in this case. For a more general bounded symmetric domain $\Omega$ we will always assume that $\nu$ and $p$ are chosen in such a way that $(\nu,\nu,p)$ is admissible. Clearly, this assumption also implies that $(\alpha,\nu,p)$ will always be admissible provided that $\alpha \geq \nu$. This observation will be crucial later on.

We will also need the following two simple propositions that will be used several times later on. The first one is basically a sloppy version of Jensen's inequality, but sufficient for our purposes.

\begin{prop} \label{AuxiliaryProp}
Let $n \in \N$, $p \in (1,\infty)$ and $x_1, \ldots, x_n \geq 0$. Then
\[\left(\sum_{k = 1}^n x_k\right)^p \leq n^p\sum_{k = 1}^n x_k^p.\]
\end{prop}

\begin{proof}
Obviously,
\[\left(\sum_{k = 1}^n x_k\right)^p \leq \left(n\max\limits_{k = 1, \ldots, n} x_k\right)^p = n^p \max\limits_{k = 1, \ldots, n} x_k^p \leq n^p\sum_{k = 1}^n x_k^p.\]
\end{proof}

\begin{prop} \label{AuxiliaryProp2}
Let $(U_k)_{k \in \N}$ be a sequence of measurable sets in $\Omega$ such that every $z \in \Omega$ belongs to at most $N$ of the sets $U_k$ and let $f \in L^p_{\nu}$. Then
\[\sum\limits_{k = 1}^{\infty} \int_{U_k} \abs{f(z)}^p \, \mathrm{d}v_{\nu}(z) \leq N\norm{f}^p.\]
\end{prop}

\begin{proof}
For every $k \in \N$ there is a disjoint decomposition $U_k = A_k^1 \cup \ldots \cup A_k^N$ such that the sets $(A_k^i)_{k \in \N}$ are again measurable and pairwise disjoint for every $i \in \set{1, \ldots, N}$ (see \cite[p.~2195]{Suarez} for details). Thus
\[\sum\limits_{k = 1}^{\infty} \int_{U_k} \abs{f(z)}^p \, \mathrm{d}v_{\nu}(z) = \sum_{i = 1}^N\sum\limits_{k = 1}^{\infty} \int_{A_k^i} \abs{f(z)}^p \, \mathrm{d}v_{\nu}(z) \leq \sum_{i = 1}^N \int_{\Omega} \abs{f(z)}^p \, \mathrm{d}v_{\nu}(z) = N\norm{f}^p.\qedhere\]
\end{proof}

\section{Band-dominated operators} \label{BDO}
In this section we introduce the notion of band-dominated operators. The name is chosen in analogy to the sequence space case $\ell^p(\Z)$, where band-dominated operators are in fact norm limits of infinite band matrices, see e.g.~\cite{HaLiSe, Lindner, LiSe, RaRoSi, RaRoSi2, Seidel}.

\begin{defn} \label{defn_BDO}
An operator $A \in \Lc(L^p_{\nu})$ is called a band operator if there exists a positive real number $\omega$ such that $M_fAM_g = 0$ for all $f,g \in L^{\infty}(\Omega)$ with $\dist_{\beta}(\supp f,\supp g) > \omega$. The number
\[\inf\set{\omega \in \R : M_fAM_g = 0 \text{ for all } f,g \in L^{\infty}(\Omega) \text{ with } \dist_{\beta}(\supp f,\supp g) > \omega}\]
is called the band width of $A$. An operator $A \in \Lc(L^p_{\nu})$ is called band-dominated if it is the norm limit of band operators. The set of band-dominated operators will be denoted by $\BDO^p_{\nu}$.
\end{defn}

The definition of band-dominated operators can be extended to operators acting on the Bergman space $A^p_{\nu}$. If $(\alpha,\nu,p)$ is admissible and $Q_{\alpha} := I - P_{\alpha}$, we can consider the natural extension $\hat{A} := AP_{\alpha} + Q_{\alpha}$ of $A \in \Lc(A^p_{\nu})$. An operator acting on $A^p_{\nu}$ is then called band-dominated if its extension is band-dominated. As will be immediate, this definition does not depend on the chosen extension. In this language the main result of this section reads as follows: $\hat{A} \in \BDO^p_{\nu}$ for all $A \in \Tf_{p,\nu}$, i.e.~Toeplitz operators are band-dominated.

\begin{thm} \label{thm0}
Let $(\alpha,\nu,p)$ be an admissible triple and $A \in \Tf_{p,\nu}$. Then $\hat{A} \in \BDO^p_{\nu}$.
\end{thm}

Before we proceed with the proof of this theorem, we show some equivalent characterizations of band-dominated operators that will prove useful later on. For this we need some cut-off functions to decompose our domain $\Omega$. Let us state the following auxiliary lemma, which is due to Carlsson and Goldfarb \cite{CaGo} (see \cite[Theorem 91]{BeDra} for a more explicit version). For the unit ball Su\'arez constructed an explicit cover in \cite[Lemma 3.1]{Suarez}.

\begin{lem} \label{lem2}
There is a (smallest possible) positive integer $N$ (depending only on the bounded symmetric domain $\Omega$) such that for any $\sigma > 0$ there is a cover of $\Omega$ by Borel sets $(B_j)_{j \in \N}$ satisfying
\begin{itemize}
	\item[$(i)$] the sets $B_j$ are pairwise disjoint,
	\item[$(ii)$] every point of $\Omega$ belongs to at most $N$ of the sets $\set{z \in \Omega : \dist_{\beta}(z,B_j) \leq \sigma}$,
	\item[$(iii)$] there is a constant $C(\sigma) > 0$ such that $\diam_{\beta}(B_j) \leq C(\sigma)$ for every $j \in \N$.
\end{itemize}
\end{lem}

For every $t \in (0,1)$ let $(B_{j,t})_{j \in \N}$ be a cover of $\Omega$ that satisfies $(i)$ to $(iii)$ in Lemma \ref{lem2} in the case $\sigma = \frac{1}{t}$ and define
\[\Xi_{j,t,k} := \set{z \in \Omega : \dist_{\beta}(z,B_{j,t}) \leq \frac{k}{3t}}\]
for $j \in \N$, $t \in (0,1)$ and $k = 1,2,3$. We now construct families of uniformly Lipschitz continuous functions (partitions of unity) according to these decompositions. Let $f_{j,t} \from \Omega \to [0,1]$ be defined by
\[f_{j,t}(z) := \frac{\dist_{\beta}(z,\Omega \setminus \Xi_{j,t,1})}{\dist_{\beta}(z,B_{j,t}) + \dist_{\beta}(z,\Omega \setminus \Xi_{j,t,1})}.\]
Clearly, $\supp f_{j,t} = \Xi_{j,t,1}$ and $f_{j,t}(z) = 1$ for $z \in B_{j,t}$. Moreover, it is easy to see that
\[\abs{f_{j,t}(z) - f_{j,t}(w)} \leq \frac{\beta(z,w)}{\dist_{\beta}(B_{j,t},\Omega \setminus \Xi_{j,t,1})} = 3t\beta(z,w).\]
Define $g_t \from \Omega \to \R$ by $g_t(z) := \sum\limits_{j = 1}^{\infty} f_{j,t}(z)$ and $\varphi_{j,t} \from \Omega \to [0,1]$ by $\varphi_{j,t} := \frac{f_{j,t}}{g_t}$. The functions $\varphi_{j,t}$ then satisfy the following properties
\begin{itemize}
	\item[$(i)$] $\sum\limits_{j = 1}^{\infty} \varphi_{j,t}(z) = 1$ for all $z \in \Omega$,
	\item[$(ii)$] $\supp \varphi_{j,t} = \Xi_{j,t,1}$ for all $j \in \N$, $t \in (0,1)$,
	\item[$(iii)$] $\abs{\varphi_{j,t}(z) - \varphi_{j,t}(w)} \leq 6Nt\beta(z,w)$ for all $w,z \in \Omega$, $j \in \N$ and $t \in (0,1)$.
\end{itemize}
Similarly, we can define functions $\psi_{j,t} \from \Omega \to [0,1]$ with the following properties:
\begin{itemize}
	\item[$(i)$] $\psi_{j,t}(z) = 1$ for all $z \in \Xi_{j,t,2}$, $j \in \N$ and $t \in (0,1)$,
	\item[$(ii)$] $\supp \psi_{j,t} = \Xi_{j,t,3}$ for all $j \in \N$ and $t \in (0,1)$,
	\item[$(iii)$] $\abs{\psi_{j,t}(z) - \psi_{j,t}(w)} \leq 3t\beta(z,w)$ for all $w,z \in \Omega$, $j \in \N$ and $t \in (0,1)$.
\end{itemize}
In particular, we have $\varphi_{j,t}\psi_{j,t} = \varphi_{j,t}$ for all $j \in \N$ and $t \in (0,1)$.

\begin{prop} \label{BDO_characterization}
Let $A \in \Lc(L^p_{\nu})$ and $N \in \N$ as in Lemma \ref{lem2}. Moreover, let $\varphi_{j,t}$ and $\psi_{j,t}$ be defined as above. Then the following are equivalent:
\begin{itemize}
	\item[$(i)$] $A$ is band-dominated,
	\item[$(ii)$] $\lim\limits_{t \to 0} \sup\limits_{\norm{f} = 1} \sum\limits_{j = 1}^{\infty} \norm{M_{a_{j,t}}AM_{1-b_{j,t}}f}^p = 0$ for all families of functions $a_{j,t},b_{j,t} \from \Omega \to [0,1]$ that satisfy $\lim\limits_{t \to 0} \inf\limits_{j \in \N} \dist_{\beta}(\supp a_{j,t},\supp(1-b_{j,t})) = \infty$ and for every $t \in (0,1)$ and $z \in \Omega$ the sets $\set{j \in \N : z \in \supp a_{j,t}}$ and $\set{j \in \N : z \in \supp b_{j,t}}$ contain at most $N$ elements,
	\item[$(iii)$] $\lim\limits_{t \to 0} \norm{\sum\limits_{j = 1}^{\infty} M_{a_{j,t}}AM_{1-b_{j,t}}} = 0$ under the same assumptions as in $(ii)$,
	\item[$(iv)$] $\lim\limits_{t \to 0} \sup\limits_{\norm{f} = 1} \sum\limits_{j = 1}^{\infty} \norm{M_{\varphi_{j,t}}AM_{1-\psi_{j,t}}f}^p = 0$,
	\item[$(v)$] $\lim\limits_{t \to 0} \norm{\sum\limits_{j = 1}^{\infty} M_{\varphi_{j,t}}AM_{1-\psi_{j,t}}} = 0$.
\end{itemize}
\end{prop}

\begin{proof}
Let $A \in \BDO^p_{\nu}$. Then there is a sequence $(A_n)_{n \in \N}$ of band operators such that $A_n \to A$. Let $\epsilon > 0$ and choose $n$ sufficiently large such that $\norm{A-A_n} < \epsilon$. Now choose $t$ sufficiently small such that $\inf\limits_{j \in \N} \dist_{\beta}(\supp a_{j,t},\supp(1-b_{j,t}))$ is larger that the band-width of $A_n$. This implies that $M_{a_{j,t}}A_nM_{1-b_{j,t}} = 0$ for all $j \in \N$. Therefore
\begin{align*}
\sum\limits_{j = 1}^{\infty} \norm{M_{a_{j,t}}AM_{1-b_{j,t}}f}^p &=
\sum\limits_{j = 1}^{\infty} \norm{M_{a_{j,t}}(A-A_n)M_{1-b_{j,t}}f}^p\\
&\leq 2^p\sum\limits_{j = 1}^{\infty} \left(\norm{M_{a_{j,t}}(A-A_n)f}^p + \norm{M_{a_{j,t}}(A-A_n)M_{b_{j,t}}f}^p\right)\\
&\leq 2^{p+1}N\epsilon^p\norm{f}^p
\end{align*}
for all $f \in L^p_{\nu}$ and sufficiently small $t$ by Proposition \ref{AuxiliaryProp2}. As $\epsilon$ was arbitrary, $(ii)$ follows.

Now assume
\[\lim\limits_{t \to 0} \sup\limits_{\norm{f} = 1} \sum\limits_{j = 1}^{\infty} \norm{M_{a_{j,t}}AM_{1-b_{j,t}}f}^p = 0.\] 
Then
\begin{align*}
\norm{\sum\limits_{j = 1}^{\infty} M_{a_{j,t}}AM_{1-b_{j,t}}f}^p &= \int_{\Omega} \abs{\sum\limits_{j = 1}^{\infty} M_{a_{j,t}}AM_{1-b_{j,t}}f}^p \, \mathrm{d}v_{\nu}\\
&\leq \int_{\Omega} N^p\sum\limits_{j = 1}^{\infty} \abs{M_{a_{j,t}}AM_{1-b_{j,t}}f}^p \, \mathrm{d}v_{\nu}\\
&= N^p\sum\limits_{j = 1}^{\infty}\int_{\Omega} \abs{M_{a_{j,t}}AM_{1-b_{j,t}}f}^p \, \mathrm{d}v_{\nu}\\
&= N^p\sum\limits_{j = 1}^{\infty} \norm{M_{a_{j,t}}AM_{1-b_{j,t}}f}^p
\end{align*}
for all $f \in L^p_{\nu}$ and $t \in (0,1)$ by Proposition \ref{AuxiliaryProp} (for every $z \in \Omega$ the sum over $j$ in the first line contains at most $N$ non-zero terms). It follows
\[\lim\limits_{t \to 0} \norm{\sum\limits_{j = 1}^{\infty} M_{a_{j,t}}AM_{1-b_{j,t}}} = 0.\]
Similarly, $(iv)$ implies $(v)$.

That $(ii)$ implies $(iv)$ and $(iii)$ implies $(v)$ is clear since $\dist_{\beta}(\supp \varphi_{j,t},\supp(1 - \psi_{j,t})) \geq \frac{1}{3t}$ and every $z \in \Omega$ belongs to at most $N$ of the sets $\supp \varphi_{j,t}$ and $\supp \psi_{j,t}$ by construction (see Lemma \ref{lem2}).

We are thus left with the assertion that $(v)$ implies $(i)$. Define
\[A_n := \sum\limits_{j = 1}^{\infty} M_{\varphi_{j,\frac{1}{n}}}AM_{\psi_{j,\frac{1}{n}}}.\]
It is easily seen that this defines a bounded linear operator (see also Lemma \ref{lem0.5} below). Moreover, $A_n$ is obviously a band operator of band width at most $C(n) + 2n$, where $C(n)$ is the constant from Lemma \ref{lem2} $(iii)$. As $\sum\limits_{j = 1}^{\infty} M_{\varphi_{j,t}} = I$ for all $t \in (0,1)$, we obtain
\[\norm{A-A_n} = \norm{\sum\limits_{j = 1}^{\infty} (M_{\varphi_{j,\frac{1}{n}}}A - M_{\varphi_{j,\frac{1}{n}}}AM_{\psi_{j,\frac{1}{n}}})} = \norm{\sum\limits_{j = 1}^{\infty} M_{\varphi_{j,\frac{1}{n}}}AM_{1-\psi_{j,\frac{1}{n}}}}\]
and this tends to $0$ as $n \to \infty$ by assumption.
\end{proof}

\begin{cor} \label{BDO_characterization_cor}
Let $A \in \BDO^p_{\nu}$ and let $a_{j,t},b_{j,t} \from \Omega \to [0,1]$ satisfy
\[\lim\limits_{t \to 0} \inf\limits_{j \in \N} \dist_{\beta}(\supp a_{j,t},\supp(1-b_{j,t})) = \infty\]
and for every $t \in (0,1)$ and $z \in \Omega$ the sets $\set{j \in \N : z \in \supp a_{j,t}}$ and $\set{j \in \N : z \in \supp b_{j,t}}$ contain at most $N$ elements. Then
\[\lim\limits_{t \to 0} \sup\limits_{j \in \N} \norm{M_{a_{j,t}}AM_{1-b_{j,t}}} = 0 = \lim\limits_{t \to 0} \sup\limits_{j \in \N} \norm{M_{1-b_{j,t}}AM_{a_{j,t}}}.\]
\end{cor}

\begin{proof}
As
\[\norm{M_{a_{j,t}}AM_{1-b_{j,t}}}^p = \sup\limits_{\norm{f} = 1} \norm{M_{a_{j,t}}AM_{1-b_{j,t}}f}^p \leq \sup\limits_{\norm{f} = 1} \sum\limits_{j = 1}^{\infty} \norm{M_{a_{j,t}}AM_{1-b_{j,t}}f}^p\]
the first limit follows directly from Proposition \ref{BDO_characterization}. To show the second limit we may either repeat the first part of the proof of Proposition \ref{BDO_characterization} with the reversed ordering or just observe that $A \in \BDO^p_{\nu}$ is equivalent to $A^* \in \BDO^q_{\nu}$ for $\frac{1}{p} + \frac{1}{q} = 1$.
\end{proof}

The following characterization will also prove itself useful. The proof is quite similar to the proof of Theorem 2.1.6 in \cite{RaRoSi2}. We use the standard notation $[A,B] := AB-BA$ for the commutator of two operators $A$ and $B$.

\begin{prop} \label{BDO_characterization2}
Let $A \in \Lc(L^p_{\nu})$. Then $A$ is band-dominated if and only if
\begin{equation} \label{eq_BDO_characterization2}
\lim\limits_{t \to 0} \sup\limits_{\norm{f} = 1} \sum\limits_{j = 1}^{\infty} \norm{[A,M_{\varphi_{j,t}}]f}^p = 0.
\end{equation}
Moreover, $\lim\limits_{t \to 0} \sup\limits_{j \in \N} \norm{[A,M_{\varphi_{j,t}}]} = 0$ in this case.
\end{prop}

We divide the proof in two parts. In the first part we deal with band operators only and show a little bit more than we need here. This will come in handy later on.

\begin{lem} \label{lem3}
Let $\omega > 0$ and let $a_{j,t} \from \Omega \to [0,1]$ be measurable functions for $j \in \N$ and $t \in (0,1)$. If $\lim\limits_{t \to 0} \inf\limits_{j \in \N} \dist_{\beta}(a_{j,t}^{-1}(U),a_{j,t}^{-1}(V)) \to \infty$ for all sets $U,V \subset [0,1]$ with $\dist(U,V) > 0$, then for every $\epsilon > 0$ there exists a $t_0 > 0$ such that for all $t < t_0$ and all band operators of band width at most $\omega$ the estimate
\[\sup\limits_{j \in \N} \norm{[A,M_{a_{j,t}}]} \leq 3\norm{A}\epsilon\]
holds.
\end{lem}

\begin{proof}
Let $A \in \Lc(L^p_{\nu})$ be a band operator of band width at most $\omega$, fix $\epsilon > 0$ and set $m := \left\lceil\frac{1}{\epsilon}\right\rceil$. For $k = 1, \ldots, m$ we define the following sets:
\[U_{k,t}^j := \set{z \in \Omega : a_{j,t}(z) \geq k\epsilon} \quad \text{and} \quad V_{k,t}^j := \set{z \in \Omega : a_{j,t}(z) \geq \left(k - \frac{1}{2}\right)\epsilon}.\]
Moreover, we set
\[a_{j,t}^U := \epsilon\sum\limits_{k = 1}^{m} \chi_{U_{k,t}^j} \quad \text{and} \quad a_{j,t}^V := \epsilon\sum\limits_{k = 1}^{m} \chi_{V_{k,t}^j}.\]
Clearly, for every $z \in \Omega$, $t \in (0,1)$ and $j \in \N$ either $a_{j,t}(z) < \epsilon$ or $a_{j,t}(z) \in [l\epsilon, (l+1)\epsilon)$ for some $l \in \set{1, \ldots, m}$. In either case this implies $\abs{a_{j,t}(z) - a_{j,t}^U(z)} < \epsilon$ and hence $\sup\limits_{j \in \N}\norm{a_{j,t} - a_{j,t}^U}_{\infty} < \epsilon$. Similarly, we obtain $\sup\limits_{j \in \N}\norm{a_{j,t} - a_{j,t}^V}_{\infty} < \epsilon$. This of course implies
\[\sup\limits_{j \in \N} \norm{M_{a_{j,t}} - M_{a_{j,t}^U}} < \epsilon \quad \text {and} \quad \sup\limits_{j \in \N} \norm{M_{a_{j,t}} - M_{a_{j,t}^V}} < \epsilon.\]
It follows
\begin{align} \label{ineq1}
\sup\limits_{j \in \N} \norm{AM_{a_{j,t}} - M_{a_{j,t}}A} &\leq \sup\limits_{j \in \N} \left(\norm{A(M_{a_{j,t}} - M_{a_{j,t}^V})} + \norm{(AM_{a_{j,t}^V} - M_{a_{j,t}^U}A)} + \norm{(M_{a_{j,t}^U} - M_{a_{j,t}})A}\right)\notag\\
&\leq 2\norm{A}\epsilon + \sup\limits_{j \in \N} \norm{(AM_{a_{j,t}^V} - M_{a_{j,t}^U}A)}\notag\\
&= 2\norm{A}\epsilon + \epsilon\sup\limits_{j \in \N} \norm{\sum\limits_{k = 1}^m (AM_{\chi_{V_{k,t}^j}} - M_{\chi_{U_{k,t}^j}}A)}\notag\\
&\leq 2\norm{A}\epsilon + \epsilon\sup\limits_{j \in \N} \norm{\sum\limits_{k = 1}^m M_{\chi_{\Omega \setminus U_{k,t}^j}}AM_{\chi_{V_{k,t}^j}}} + \epsilon\sup\limits_{j \in \N} \norm{\sum\limits_{k = 1}^m M_{\chi_{U_{k,t}^j}}AM_{\chi_{\Omega \setminus V_{k,t}^j}}}.
\end{align}
Since $U_{k,t}^j = a_{j,t}^{-1}([k\epsilon,1])$ and $\Omega \setminus V_{k,t}^j = a_{j,t}^{-1}([0,\left(k-\frac{1}{2}\right)\epsilon))$, we obtain
\[\inf\limits_{j \in \N}\dist_{\beta}(U_{k,t}^j,\Omega \setminus V_{k,t}^j) \to \infty\]
for all $k$ as $t \to 0$ by assumption. Choose $t$ sufficiently small such that $\dist_{\beta}(U_{k,t}^j,\Omega \setminus V_{k,t}^j) > \omega$ for all $j \in \N$. As $A$ is a band operator of band width at most $\omega$, the third term in \eqref{ineq1} vanishes.

Similarly, setting $U_{0,t}^j := \Omega$ and $V_{m+1,t}^j := \emptyset$, we get
\begin{align*}
\sup\limits_{j \in \N} \norm{\sum\limits_{k = 1}^m M_{\chi_{\Omega \setminus U_{k,t}^j}}AM_{\chi_{V_{k,t}^j}}} &\leq \sup\limits_{j \in \N}\norm{\sum\limits_{k = 1}^m M_{\chi_{\Omega \setminus U_{k,t}^j}}AM_{\chi_{V_{k+1,t}^j}}} + \sup\limits_{j \in \N} \norm{\sum\limits_{k = 1}^m M_{\chi_{\Omega \setminus U_{k-1,t}^j}}AM_{\chi_{V_{k,t}^j \setminus V_{k+1,t}^j}}}\\
&\quad + \sup\limits_{j \in \N}\norm{\sum\limits_{k = 1}^m M_{\chi_{U_{k-1,t}^j \setminus U_{k,t}^j}}AM_{\chi_{V_{k,t}^j \setminus V_{k+1,t}^j}}}\\
&= \sup\limits_{j \in \N} \norm{\sum\limits_{k = 1}^m M_{\chi_{U_{k-1,t}^j \setminus U_{k,t}^j}}AM_{\chi_{V_{k,t}^j \setminus V_{k+1,t}^j}}}.
\end{align*}
As the sets $U_{k-1,t}^j \setminus U_{k,t}^j$ are pairwise disjoint for $k = 1, \ldots, m$, this can be further estimated as
\begin{align*}
\sup\limits_{j \in \N} \norm{\sum\limits_{k = 1}^m M_{\chi_{U_{k-1,t}^j \setminus U_{k,t}^j}}AM_{\chi_{V_{k,t}^j \setminus V_{k+1,t}^j}}f}^p &= \sup\limits_{j \in \N} \sum\limits_{k = 1}^m \norm{M_{\chi_{U_{k-1,t}^j \setminus U_{k,t}^j}}AM_{\chi_{V_{k,t}^j \setminus V_{k+1,t}^j}}f}^p\\
&\leq \norm{A}^p\sup\limits_{j \in \N} \sum\limits_{k = 1}^m \norm{M_{\chi_{V_{k,t}^j \setminus V_{k+1,t}^j}}f}^p\\
&= \norm{A}^p\sup\limits_{j \in \N} \norm{M_{\chi_{V_{1,t}^j}}f}^p\\
&\leq \norm{A}^p\norm{f}^p
\end{align*}
for all $f \in L^p_{\nu}$. Thus
\[\sup\limits_{j \in \N} \norm{(AM_{a_{j,t}} - M_{a_{j,t}}A)} \leq 2\norm{A}\epsilon + \norm{A}\epsilon = 3\norm{A}\epsilon\qedhere.\]
\end{proof}

Now we can prove Proposition \ref{BDO_characterization2}.

\begin{proof}[Proof of Proposition \ref{BDO_characterization2}]
Let $A \in \BDO^p_{\nu}$ and fix $\epsilon > 0$. Then there is a sequence of band operators $(A_n)_{n \in \N}$ such that $A_n \to A$. Choose $n$ sufficiently large such that $\norm{A-A_n} < \epsilon$. Now observe that the functions $\varphi_{j,t}$ satisfy the assumption in Lemma \ref{lem3}. Indeed, let $U,V \subset [0,1]$ with $\dist(U,V) > 0$ and $w_{j,t} \in \varphi_{j,t}^{-1}(U)$, $z_{j,t} \in \varphi_{j,t}^{-1}(V)$. Then
\[\beta(z_{j,t},w_{j,t}) \geq \frac{1}{6Nt}\abs{\varphi_{j,t}(z_{j,t}) - \varphi_{j,t}(w_{j,t})} \geq \frac{1}{6Nt}\dist(U,V) \to \infty\]
as $t \to 0$. Thus there is a $t > 0$ such that
\[\sup\limits_{j \in \N} \norm{[A,M_{\varphi_{j,t}}]} \leq \sup\limits_{j \in \N} \norm{[A_n,M_{\varphi_{j,t}}]} + \sup\limits_{j \in \N} \norm{[A-A_n,M_{\varphi_{j,t}}]} \leq 3\norm{A_n}\epsilon + 2\epsilon \leq 3(\norm{A}+\epsilon)\epsilon + 2\epsilon\]
by Lemma \ref{lem3}. As $\epsilon$ was arbitrary, this implies
\[\lim\limits_{t \to 0} \sup\limits_{j \in \N} \norm{[A,M_{\varphi_{j,t}}]} = 0.\]
Using
\begin{align*}
\sup\limits_{\norm{f} = 1} \sum\limits_{j = 1}^{\infty} \norm{[A,M_{\varphi_{j,t}}]f}^p &\leq 2^p\sup\limits_{\norm{f} = 1} \sum\limits_{j = 1}^{\infty} \left(\norm{[A,M_{\varphi_{j,t}}]M_{\psi_{j,t}}f}^p + \norm{[A,M_{\varphi_{j,t}}]M_{1-\psi_{j,t}}f}^p\right)\\
&\leq 2^p\sup\limits_{\norm{f} = 1} \sum\limits_{j = 1}^{\infty} \left(\norm{[A,M_{\varphi_{j,t}}]}^p\norm{M_{\psi_{j,t}}f}^p + \norm{M_{\varphi_{j,t}}AM_{1-\psi_{j,t}}f}^p\right)\\
&\leq 2^pN\sup\limits_{j \in \N} \norm{[A,M_{\varphi_{j,t}}]}^p + \sup\limits_{\norm{f} = 1} \sum\limits_{j = 1}^{\infty} \norm{M_{\varphi_{j,t}}AM_{1-\psi_{j,t}}f}^p,
\end{align*}
and Proposition \ref{BDO_characterization}, we obtain
\[\lim\limits_{t \to 0} \sup\limits_{\norm{f} = 1} \sum\limits_{j = 1}^{\infty} \norm{[A,M_{\varphi_{j,t}}]f}^p = 0\]
as claimed.

Conversely, assume that \eqref{eq_BDO_characterization2} holds. Clearly, this implies $\lim\limits_{t \to 0} \sup\limits_{j \in \N} \norm{[A,M_{\varphi_{j,t}}]} = 0$ as well (cf.~proof of Corollary \ref{BDO_characterization_cor}). We can thus proceed as above to obtain
\begin{align*}
\sup\limits_{\norm{f} = 1}\sum\limits_{j = 1}^{\infty} \norm{M_{\varphi_{j,t}}AM_{1-\psi_{j,t}}f}^p &\leq 2^p\sup\limits_{\norm{f} = 1}\sum\limits_{j = 1}^{\infty} \left(\norm{[M_{\varphi_{j,t}},A]f}^p + \norm{[M_{\varphi_{j,t}},A]M_{\psi_{j,t}}f}^p\right)\\
&\leq 2^p\sup\limits_{\norm{f} = 1}\sum\limits_{j = 1}^{\infty} (\norm{[M_{\varphi_{j,t}},A]f}^p + \norm{[M_{\varphi_{j,t}},A]}^p\norm{M_{\psi_{j,t}}f}^p)\\
&\leq 2^p\sup\limits_{\norm{f} = 1}\sum\limits_{j = 1}^{\infty} \norm{[M_{\varphi_{j,t}},A]f}^p + 2^pN\sup\limits_{j \in \N} \norm{[M_{\varphi_{j,t}},A]}^p.
\end{align*}
and by assumption, this tends to $0$ as $t \to 0$. Thus $A \in \BDO^p_{\nu}$ by Proposition \ref{BDO_characterization}.
\end{proof}

Here are some algebraic properties of $\BDO^p_{\nu}$:

\begin{prop} \label{prop0}
$\BDO^p_{\nu}$ has the following properties:
\begin{itemize}
	\item[$(i)$] It holds $M_f \in \BDO^p_{\nu}$ for all $f \in L^{\infty}(\Omega)$.
	\item[$(ii)$] $\BDO^p_{\nu}$ is a closed subalgebra of $\Lc(L^p_{\nu})$.
	\item[$(iii)$] If $A \in \BDO^p_{\nu}$ is Fredholm, then every regularizer $B$ of $A$ is again in $\BDO^p_{\nu}$. In particular, $\BDO^p_{\nu}$ is inverse closed.
	\item[$(iv)$] $\BDO^p_{\nu}$ contains $\Kc(L^p_{\nu})$ as a closed two-sided ideal.
	\item[$(v)$] It holds $A \in \BDO^p_{\nu} \Longleftrightarrow A^* \in \BDO^q_{\nu}$ for $\frac{1}{p} = \frac{1}{q} = 1$. In particular, $\BDO^2_{\nu}$ is a $C^*$-algebra.
\end{itemize}
\end{prop}

\begin{proof}
$(i)$, $(ii)$ and $(v)$ are easy to see.

$(iii)$: Set $\tilde{\varphi}_{j,t} := \varphi_{j,t} - \varphi_{j,t}(0)$ for all $j \in \N$ and $t \in (0,1)$. If $A$ is Fredholm, there exist a regularizer $B \in \Lc(L^p_{\nu})$ and compact operators $K_1,K_2 \in \Kc(L^p_{\nu})$ such that $AB = I + K_1$ and $BA = I + K_2$. This implies
\[[B,M_{\varphi_{j,t}}] = [B,M_{\tilde{\varphi}_{j,t}}] = B[M_{\tilde{\varphi}_{j,t}},A]B - BM_{\tilde{\varphi}_{j,t}}K_1 + K_2M_{\tilde{\varphi}_{j,t}}B.\]
Thus
\begin{equation} \label{eq_Reg_BDO}
\sup\limits_{\norm{f} = 1} \sum\limits_{j = 1}^{\infty} \norm{[B,M_{\varphi_{j,t}}]f}^p \leq 3^p \sup\limits_{\norm{f} = 1} \sum\limits_{j = 1}^{\infty} \left(\norm{B[M_{\varphi_{j,t}},A]Bf}^p + \norm{BM_{\tilde{\varphi}_{j,t}}K_1f}^p + \norm{K_2M_{\tilde{\varphi}_{j,t}}Bf}^p\right).
\end{equation}
As $A$ is band-dominated, the first term tends to $0$ as $t \to 0$ by Proposition \ref{BDO_characterization2}. To estimate the other two terms, define $D(0,R) := \set{z \in \Omega : \beta(0,z) < R}$ and $D(0,R)^c := \Omega \setminus D(0,R)$ for $R > 0$. Since $K_1$ and $K_2$ are compact, $\|M_{\chi_{D(0,R)^c}}K_1\|$ and $\|K_2M_{\chi_{D(0,R)^c}}\|$ tend to $0$ as $R \to \infty$. Moreover, as for fixed $t$ the origin is contained in at most $N$ of the sets $\Xi_{j,t,3} = \set{z \in \Omega : \dist_{\beta}(z,B_{j,t}) \leq \frac{1}{t}}$ and $\supp \varphi_{j,t} = \Xi_{j,t,1} = \set{z \in \Omega : \dist_{\beta}(z,B_{j,t}) \leq \frac{1}{3t}}$, we get that $\tilde{\varphi}_{j,t}$ vanishes on $D(0,\frac{2}{3\sqrt{t}}) \subset D(0,\frac{2}{3t})$ for all but at most $N$ integers $j$. W.l.o.g.~we may assume that these integers are $j = 1, \ldots, N$. For $j = 1, \ldots, N$ and $z \in D(0,\frac{2}{3\sqrt{t}})$ it holds $|\tilde{\varphi}_{j,t}(z)| \leq 6Nt\frac{2}{3\sqrt{t}} = 4N\sqrt{t}$ by property $(iii)$ of $\varphi_{j,t}$. Therefore
\begin{align*}
\sup\limits_{\norm{f} = 1} \sum\limits_{j = 1}^{\infty} \norm{BM_{\tilde{\varphi}_{j,t}}K_1f}^p &\leq 2^p\sup\limits_{\norm{f} = 1} \sum\limits_{j = 1}^{\infty} \left(\norm{BM_{\tilde{\varphi}_{j,t}}M_{\chi_{D(0,R)^c}}K_1f}^p + \norm{BM_{\tilde{\varphi}_{j,t}}M_{\chi_{D(0,R)}}K_1f}^p\right)\\
&\leq 2^p\norm{B}^p\|M_{\chi_{D(0,R)^c}}K_1\|^p\sup\limits_{\norm{f} = 1} \sum\limits_{j = 1}^{\infty} \norm{M_{\tilde{\varphi}_{j,t}}f}^p\\
&\quad + 2^pN\norm{B}^p(4N\sqrt{t})^p\norm{K_1}^p
\end{align*}
for $t < \frac{4}{9R^2}$. As
\[\sup\limits_{\norm{f} = 1} \sum\limits_{j = 1}^{\infty} \norm{M_{\tilde{\varphi}_{j,t}}f}^p = \sup\limits_{\norm{f} = 1} \left(\sum\limits_{j = 1}^N \norm{M_{\tilde{\varphi}_{j,t}}f}^p + \sum\limits_{j = N+1}^{\infty} \norm{M_{\varphi_{j,t}}f}^p\right)\]
is bounded by $2N$, this implies $\lim\limits_{t \to 0} \sup\limits_{\norm{f} = 1} \sum\limits_{j = 1}^{\infty} \norm{BM_{\tilde{\varphi}_{j,t}}K_1f}^p = 0$. Similarly, we obtain the equality $\lim\limits_{t \to 0} \sup\limits_{\norm{f} = 1} \sum\limits_{j = 1}^{\infty} \norm{K_2M_{\tilde{\varphi}_{j,t}}Bf}^p = 0$. Plugging these observations into \eqref{eq_Reg_BDO}, we conclude
\[\lim\limits_{t \to 0} \sup\limits_{\norm{f} = 1} \sum\limits_{j = 1}^{\infty} \norm{[B,M_{\varphi_{j,t}}]f}^p = 0\]
and hence $B \in \BDO^p_{\nu}$ by Proposition \ref{BDO_characterization2}.

$(iv)$ The argument in $(iii)$ shows $\lim\limits_{t \to 0} \sup\limits_{\norm{f} = 1} \sum\limits_{j = 1}^{\infty} \norm{M_{\tilde{\varphi}_{j,t}}Kf}^p = \lim\limits_{t \to 0} \sup\limits_{\norm{f} = 1} \sum\limits_{j = 1}^{\infty} \norm{KM_{\tilde{\varphi}_{j,t}}f}^p = 0$ for $K \in \Kc(L^p_{\nu})$. Thus the assertion follows again by Proposition \ref{BDO_characterization2}.
\end{proof}

To prove Theorem \ref{thm0}, we will only need one more auxiliary lemma. A similar result for the unit ball was shown in \cite[Lemma 3.4]{MiSuWi}.

\begin{lem} \label{lem1}
Let $(\alpha,\nu,p)$ be an admissible triple and for every $j$ let $a_j,b_j \from \Omega \to [0,1]$ be measurable functions. If
\begin{itemize}
	\item[$(i)$] there exists a $\sigma \geq 0$ such that $\dist_{\beta}(\supp a_j,\supp (1-b_j)) \geq \sigma$ for all $j \in \N$,
	\item[$(ii)$] there is an integer $N$ such that every $z \in \Omega$ belongs to at most $N$ of the sets $\supp a_j$ and to at most $N$ of the sets $\supp b_j$,
\end{itemize}
then there is a function $\sigma \mapsto \beta_{p,\alpha,\nu}(\sigma)$ (depending only on $p$, $\alpha$ and $\nu$) converging to $0$ as $\sigma \to \infty$ such that
\[\norm{\sum\limits_{j = 1}^{\infty} M_{a_j}P_{\alpha}M_{1-b_j}} \leq N\beta_{p,\alpha,\nu}(\sigma).\]
In other words,
\[\norm{\sum\limits_{j = 1}^{\infty} M_{a_j}P_{\alpha}M_{1-b_j}} \to 0\] 
as $\inf\limits_{j \in \N}\dist_{\beta}(\supp a_j,\supp(1-b_j)) \to \infty$.
\end{lem}

\begin{proof}
Let us consider the case where every $z \in \Omega$ belongs to at most $1$ of the sets $\supp a_j$ first. Define
\[\Phi(z,w) := \sum\limits_{j = 1}^{\infty} \chi_{\supp a_j}(z)\chi_{\supp(1-b_j)}(w)\abs{h(z,w)}^{-\alpha-g}.\]
Then
\begin{align*}
\abs{\sum\limits_{j=1}^{\infty} (M_{a_j}P_{\alpha}M_{1-b_j}f)(z)} &= \abs{\sum\limits_{j=1}^{\infty} a_j(z) \int_{\Omega} (1-b_j(w))f(w)h(z,w)^{-\alpha-g} \, \mathrm{d}v_{\alpha}(w)}\\
&\leq \frac{c_{\alpha}}{c_{\nu}}\int_{\Omega} \Phi(z,w)h(w,w)^{\alpha-\nu}\abs{f(w)} \, \mathrm{d}v_{\nu}(w).
\end{align*}
As in the proof of Proposition \ref{BoundedProjections}, we want to apply Schur's test with $h(z) := h(z,z)^s$, where $s \in \R$ has to be determined later. We thus need to show that there exist two constants $C_1$ and $C_2$ such that
\[(I1) := \int_{\Omega} \Phi(z,w)h(w,w)^{\alpha-\nu}h(w,w)^{sq} \, \mathrm{d}v_{\nu}(w) \leq C_1h(z,z)^{sq}\]
for (almost) every $z \in \Omega$ and
\[(I2) := \int_{\Omega} \Phi(z,w)h(w,w)^{\alpha-\nu}h(z,z)^{sp} \, \mathrm{d}v_{\nu}(z) \leq C_2h(w,w)^{sp}\]
for (almost) every $w \in \Omega$. 

So let $z \in \Omega$. We may assume that $z \in \supp a_j$ for some $j \in \N$, otherwise the left-hand side is just $0$. Now choose $s \in (-\frac{\alpha+1}{q},-\frac{(r-1)a}{2q}) \cap (-\frac{\nu+1}{p},-\frac{\nu}{p} + \frac{\alpha}{p} - \frac{(r-1)a}{2p})$ as in the proof of Proposition \ref{BoundedProjections} and $q_0 > 1$ sufficiently small such that $(sq+\alpha)q_0 > -1$ and $(sq+g)q_0 - g < -\frac{(r-1)a}{2}$. Moreover, let $D(w,r) := \set{z \in \Omega : \beta(w,z) < r}$ for midpoints $w$ and radii $r$. Since we assumed $N = 1$, there is only one term contributing to $\Phi(z,w)$ and so
\begin{align*}
(I1) &= c_{\nu}\int_{\Omega} \Phi(z,w)h(w,w)^{sq+\alpha} \, \mathrm{d}v(w)\\
&= c_{\nu}\int_{\supp(1-b_j)} \abs{h(z,w)}^{-\alpha-g}h(w,w)^{sq+\alpha} \, \mathrm{d}v(w)\\
&\leq c_{\nu}\int_{\Omega \setminus D(z,\sigma)} \abs{h(z,w)}^{-\alpha-g}h(w,w)^{sq+\alpha} \, \mathrm{d}v(w)\\
&= c_{\nu}\int_{\Omega \setminus D(0,\sigma)} \abs{h(\phi_z(0),\phi_z(u))}^{-\alpha-g}h(\phi_z(u),\phi_z(u))^{sq+\alpha} \, \mathrm{d}\phi_z(u)\\
&= c_{\nu}\int_{\Omega \setminus D(0,\sigma)} \abs{\frac{h(z,z)}{h(z,u)}}^{-\alpha-g}\left(\frac{h(z,z)h(u,u)}{\abs{h(z,u)}^2}\right)^{sq+\alpha} \frac{h(z,z)^g}{\abs{h(z,u)}^{2g}} \, \mathrm{d}v(u)\\
&= c_{\nu}h(z,z)^{sq}\int_{\Omega \setminus D(0,\sigma)} \abs{h(z,u)}^{-2sq-\alpha-g}h(u,u)^{sq+\alpha} \, \mathrm{d}v(u)\\
&\leq c_{\nu}h(z,z)^{sq}\abs{\Omega \setminus D(0,\sigma)}^{1/p_0}\left(\int_{\Omega} \abs{h(z,u)}^{-(2sq+\alpha+g)q_0}h(u,u)^{(sq+\alpha)q_0} \, \mathrm{d}v(u)\right)^{1/q_0}\\
&\leq c_{\nu}h(z,z)^{sq}\abs{\Omega \setminus D(0,\sigma)}^{1/p_0}C^{1/q_0},
\end{align*}
where $C$ is some constant (coming from the Rudin-Forelli estimates \cite[Theorem 4.1]{FaKo}) and $\frac{1}{p_0} + \frac{1}{q_0} = 1$ as usual.

Now let $w \in \Omega$. We obtain
\begin{align*}
(I2) &= c_{\nu}h(w,w)^{\alpha-\nu}\int_{\Omega} \Phi(z,w)h(z,z)^{sp+\nu} \, \mathrm{d}v(z)\\
&\leq c_{\nu}h(w,w)^{\alpha-\nu}\int_{\Omega} \sum_{j = 1}^{\infty} \chi_{\supp a_j}(z)\abs{h(z,w)}^{-\alpha-g}h(z,z)^{sp+\nu} \, \mathrm{d}v(z)\\
&\leq c_{\nu}h(w,w)^{\alpha-\nu}\int_{\Omega} \abs{h(z,w)}^{-\alpha-g}h(z,z)^{sp+\nu} \, \mathrm{d}v(z)\\
&\leq C_2h(w,w)^{sp}
\end{align*}
for $s \in (-\frac{\nu+1}{p},-\frac{\nu}{p} + \frac{\alpha}{p} - \frac{(r-1)a}{2p}))$ as in Proposition \ref{BoundedProjections}. Thus, by Schur's test, we have the following norm estimate:
\[\norm{\sum\limits_{j=1}^{\infty} M_{a_j}P_{\alpha}M_{1-b_j}} \leq \left(c_{\nu}\abs{\Omega \setminus D(0,\sigma)}^{1/p_0}C^{1/q_0}\right)^{1/q}C_2^{1/p} =: \beta_{p,\alpha,\nu}(\sigma)\]
with $\beta_{p,\alpha,\nu}(\sigma) \to 0$ as $\sigma \to \infty$. This proves the estimate in the case $N = 1$.

Now let us consider the case $N > 1$. As in the proof of Proposition \ref{AuxiliaryProp2}, there is a disjoint decomposition $\supp a_j = A_j^1 \cup \ldots \cup A_j^N$ such that the sets $(A_j^i)_{j \in \N}$ are again measurable and pairwise disjoint for every $i \in \set{1, \ldots, N}$. It follows
\begin{align*}
\norm{\sum\limits_{j=1}^{\infty} M_{a_j}P_{\alpha}M_{1-b_j}} &= \norm{\sum\limits_{j=1}^{\infty}\sum\limits_{i=1}^N M_{a_j\chi_{A_j^i}}P_{\alpha}M_{1-b_j}}\\
&\leq\sum\limits_{i=1}^N \norm{\sum\limits_{j=1}^{\infty}M_{a_j\chi_{A_j^i}}P_{\alpha}M_{1-b_j}}\\
&\leq N\beta_{p,\alpha,\nu}(\sigma)\qedhere.
\end{align*}
\end{proof}

\begin{proof}[Proof of Theorem \ref{thm0}]
Combining Lemma \ref{lem1} and Proposition \ref{BDO_characterization}, we get that $P_{\alpha}$ is band-dominated. By Proposition \ref{prop0} the set $\BDO^p_{\nu}$ is a Banach algebra that contains all multiplication operators. It thus contains all operators of the form $T_fP_{\alpha}+Q_{\alpha} = P_{\nu}M_fP_{\alpha} + Q_{\alpha}$ and therefore all operators of the form $AP_{\alpha}+Q_{\alpha}$ with $A \in \Tf_{p,\nu}$.
\end{proof}

\section{A Fredholm criterion for band-dominated operators} \label{Fredholm_criterion}

In this rather short section we show a Fredholm criterion for band-dominated operators. First, we need another auxiliary proposition that is of course well-known. For completeness we include a short proof.

\begin{prop} \label{prop2.5}
Let $(\alpha,\nu,p)$ be an admissible triple and let $D \subset \Omega$ be a compact set. Then the operators $P_{\alpha}M_{\chi_D}$ and $M_{\chi_D}P_{\alpha} \from L^p_{\nu} \to L^p_{\nu}$ are compact.
\end{prop}

\begin{proof}
By definition,
\[(P_{\alpha}M_{\chi_D}f)(z) = c_{\alpha}\int_{\Omega} \chi_D(w)f(w)h(z,w)^{-\alpha-g} h(w,w)^{\alpha} \, \mathrm{d}v(w)\]
for all $f \in L^p_{\nu}$. As $D$ is compact, $\chi_D(w)h(z,w)^{-\alpha-g} h(w,w)^{\alpha}$ is uniformly bounded (shown for example in the proof of \cite[Proposition 3]{Englis2}). This implies that $P_{\alpha}M_{\chi_{B_R}}$ is compact by the Hille-Tamarkin theorem (see e.g.~\cite[Theorem 41.6]{Zaanen}).

Similarly,
\[(M_{\chi_D}P_{\alpha}f)(z) = c_{\alpha}\int_{\Omega} \chi_D(z)f(w)h(z,w)^{-\alpha-g} h(w,w)^{\alpha} \, \mathrm{d}v(w)\]
and thus $M_{\chi_D}P_{\alpha}$ is compact by the same argument.
\end{proof}

We will also need the following lemma, which is a small modification of \cite[Proposition 13]{RaRoSi}.

\begin{lem} \label{lem0.5}
For $j \in \N$ let $a_j, b_j \from \Omega \to [0,1]$ be measurable functions and $A_j \in \Lc(L^p_{\nu})$ so that the sequence $(A_j)_{j \in \N}$ is uniformly bounded. If there is an integer $N$ such that every $z \in \Omega$ belongs to at most $N$ of the sets $\supp a_j$ and to at most $N$ of the sets $\supp b_j$, then the series $\sum\limits_{j = 1}^{\infty} M_{a_j}A_jM_{b_j}$ converges $*$-strongly and
\[\norm{\sum\limits_{j = 1}^{\infty} M_{a_j}A_jM_{b_j}} \leq N^2\sup\limits_{j \in \N} \norm{A_j}.\]
Moreover,
\[\sum\limits_{j = 1}^{\infty} \norm{M_{a_j}A_jM_{b_j}f}^p \leq N\sup\limits_{j \in \N} \norm{A_j}^p\norm{f}^p\]
for all $f \in L^p_{\nu}$.
\end{lem}

\begin{proof}
Since every $z \in \Omega$ belongs to at most $N$ of the sets $\supp a_j$, it follows
\begin{align*}
\norm{\sum\limits_{j = 1}^{\infty} M_{a_j}A_jM_{b_j}f}^p &= \int_{\Omega} \abs{\sum\limits_{j = 1}^{\infty} M_{a_j}A_jM_{b_j}f}^p \, \mathrm{d}v_{\nu}\\
&\leq \int_{\Omega} N^p\sum\limits_{j = 1}^{\infty} \abs{M_{a_j}A_jM_{b_j}f}^p \, \mathrm{d}v_{\nu}\\
&= N^p\sum\limits_{j = 1}^{\infty}\int_{\Omega} \abs{M_{a_j}A_jM_{b_j}f}^p \, \mathrm{d}v_{\nu}\\
&= N^p\sum\limits_{j = 1}^{\infty} \norm{M_{a_j}A_jM_{b_j}f}^p
\end{align*}
for all $f \in L^p_{\nu}$ by Proposition \ref{AuxiliaryProp} (for every $z \in \Omega$ the sum over $j$ in the first line contains at most $N$ non-zero terms). Using Proposition \ref{AuxiliaryProp2}, we also get $\sum\limits_{j = 1}^{\infty} \norm{M_{b_j}f}^p \leq N\norm{f}^p$. We thus obtain
\[\norm{\sum\limits_{j = 1}^{\infty} M_{a_j}A_jM_{b_j}f}^p \leq N^p\sum\limits_{j = 1}^{\infty} \norm{M_{a_j}A_jM_{b_j}f}^p \leq N^p\sum\limits_{j = 1}^{\infty} \norm{A_j}^p\norm{M_{b_j}f}^p \leq N^{p+1}\sup\limits_{j \in \N} \norm{A_j}^p\norm{f}^p\]
for all $f \in L^p_{\nu}$. This yields both inequalities and the $*$-strong convergence follows easily as well.
\end{proof}

\begin{prop} \label{prop3}
Let $(\alpha,\nu,p)$ be an admissible triple, $A \in \BDO^p_{\nu}$ satisfy $[A,P_{\alpha}] = 0$ and $\psi_{j,t}$ as above. If there is a constant $M > 0$ such that for every $t \in (0,1)$ there is a $j_0 \in \N$ such that for all $j \geq j_0$ there are operators $B_{j,t},C_{j,t} \in \Lc(L^p_{\nu})$ with $\norm{B_{j,t}},\norm{C_{j,t}} \leq M$ and
\[B_{j,t}AM_{\psi_{j,t}} = M_{\psi_{j,t}} = M_{\psi_{j,t}}AC_{j,t},\]
then $A|_{A^p_{\nu}} \in \Lc(A^p_{\nu})$ is Fredholm and $\norm{(A|_{A^p_{\nu}} + \Kc(A^p_{\nu}))^{-1}} \leq 2\min\set{\norm{P_{\nu}},\norm{P_{\alpha}}}N^2M$.
\end{prop}

\begin{proof}
Let the functions $\varphi_{j,t}$ be as above and $\epsilon > 0$. Then by Lemma \ref{lem0.5}, the series
\[B_t := \sum\limits_{j = j_0}^{\infty} M_{\psi_{j,t}}B_{j,t}M_{\varphi_{j,t}}\]
converges strongly with $\norm{B_t} \leq N^2M$. Multiplying by $A$, we obtain
\begin{align} \label{eq1}
B_tA = \sum\limits_{j = j_0}^{\infty} M_{\psi_{j,t}}B_{j,t}M_{\varphi_{j,t}}AM_{\psi_{j,t}} + \sum\limits_{j = j_0}^{\infty} M_{\psi_{j,t}}B_{j,t}M_{\varphi_{j,t}}AM_{1-\psi_{j,t}},
\end{align}
where the strong convergence of the two series on the right-hand side is again guaranteed by Lemma \ref{lem0.5}. As every $z \in \Omega$ belongs to at most $N$ of the sets $\supp \psi_{j,t}$, the same computation as in the proof of Lemma \ref{lem0.5} yields
\begin{align*}
\norm{\sum\limits_{j = j_0}^{\infty} M_{\psi_{j,t}}B_{j,t}M_{\varphi_{j,t}}AM_{1-\psi_{j,t}}f}^p &\leq N^p\sum\limits_{j = j_0}^{\infty} \norm{M_{\psi_{j,t}}B_{j,t}M_{\varphi_{j,t}}AM_{1-\psi_{j,t}}f}^p\\
&\leq N^pM^p\sum\limits_{j = j_0}^{\infty} \norm{M_{\varphi_{j,t}}AM_{1-\psi_{j,t}}f}^p
\end{align*}
for every $f \in L^p_{\nu}$. Therefore the second term in \eqref{eq1} tends to $0$ by Proposition \ref{BDO_characterization}. For the first term we further compute
\[\sum\limits_{j = j_0}^{\infty} M_{\psi_{j,t}}B_{j,t}M_{\varphi_{j,t}}AM_{\psi_{j,t}} = \sum\limits_{j = j_0}^{\infty} M_{\psi_{j,t}}B_{j,t}AM_{\varphi_{j,t}}M_{\psi_{j,t}} + \sum\limits_{j = j_0}^{\infty} M_{\psi_{j,t}}B_{j,t}[M_{\varphi_{j,t}},A]M_{\psi_{j,t}},\]
where the latter term tends to $0$ by Lemma \ref{lem0.5} and Proposition \ref{BDO_characterization2}. Furthermore, we have
\[\sum\limits_{j = j_0}^{\infty} M_{\psi_{j,t}}B_{j,t}AM_{\varphi_{j,t}}M_{\psi_{j,t}} = \sum\limits_{j = j_0}^{\infty} M_{\psi_{j,t}}B_{j,t}AM_{\psi_{j,t}}M_{\varphi_{j,t}} = \sum\limits_{j = j_0}^{\infty} M_{\psi_{j,t}}M_{\psi_{j,t}}M_{\varphi_{j,t}} = \sum\limits_{j = j_0}^{\infty} M_{\varphi_{j,t}}.\]
Combining all these estimates, we conclude 
\[\lim\limits_{t \to 0} \norm{B_tA - \sum\limits_{j = j_0}^{\infty} M_{\varphi_{j,t}}} = 0.\]
In particular, we have
\[\lim\limits_{t \to 0} \norm{P_{\nu}B_tA|_{A^p_{\nu}} - \sum\limits_{j = j_0}^{\infty} P_{\nu}M_{\varphi_{j,t}}|_{A^p_{\nu}}} \leq \lim\limits_{t \to 0} \norm{P_{\nu}}\norm{B_tA - \sum\limits_{j = j_0}^{\infty} M_{\varphi_{j,t}}} = 0.\]
Now as the functions $\varphi_{j,t}$ have compact support, the operators $P_{\nu}M_{\varphi_{j,t}}|_{A^p_{\nu}}$ are compact by Proposition \ref{prop2.5} and hence $\sum\limits_{j = j_0}^{\infty} P_{\nu}M_{\varphi_{j,t}}|_{A^p_{\nu}} \in I + \Kc(A^p_{\nu})$ for every $t > 0$. We deduce that $P_{\nu}B_tA|_{A^p_{\nu}} + \Kc(A^p_{\nu})$ converges to $I + \Kc(A^p_{\nu})$. By a Neumann series argument, this implies that there exists a $B \in \Lc(A^p_{\nu})$ such that $BA|_{A^p_{\nu}} \in I + \Kc(A^p_{\nu})$ and 
\[\norm{B + \Kc(A^p_{\nu})} \leq 2\norm{P_{\nu}}\norm{B_t} \leq 2\norm{P_{\nu}}N^2M.\]

As $A^* \in \BDO^q_{\nu}$, we can apply the above to $A^*$ to obtain
\[\lim\limits_{t \to 0} \norm{AC_t - \sum\limits_{j = j_0}^{\infty} M_{\varphi_{j,t}}} = \lim\limits_{t \to 0} \norm{C_t^*A^* - \sum\limits_{j = j_0}^{\infty} M_{\varphi_{j,t}}} = 0\]
for
\[C_t := \sum\limits_{j = j_0}^{\infty} M_{\varphi_{j,t}}C_{j,t}M_{\psi_{j,t}}.\]
This implies
\[\lim\limits_{t \to 0} \norm{AP_{\alpha}C_t|_{A^p_{\nu}} - \sum\limits_{j = j_0}^{\infty} P_{\alpha}M_{\varphi_{j,t}}|_{A^p_{\nu}}} \leq \lim\limits_{t \to 0} \norm{P_{\alpha}}\norm{AC_t - \sum\limits_{j = j_0}^{\infty} M_{\varphi_{j,t}}} = 0\]
because $[A,P_{\alpha}] = 0$. Now we can precede as above to obtain an operator $C \in \Lc(A^p_{\nu})$ with 
\[\norm{C + \Kc(A^p_{\nu})} \leq 2\norm{P_{\alpha}}N^2M\]
such that $A|_{A^p_{\nu}}C \in I + \Kc(A^p_{\nu})$.
\end{proof}

\section{Limit operators (unit ball)} \label{Limit_operators}

From this section onwards we focus on the case of the unit ball $\Omega = \B^n$. The corresponding results are expected to hold for general bounded symmetric domains as well, but need some more preparations. These are postponed to future work.

As we expect the Fredholm information to be located at the boundary, we consider the following shift operators\footnote{Strictly speaking, they are rather reflections than shifts, but they serve the purpose of ``shifting'' operators to the boundary as $z \to \partial\B^n$.}. Let $U_z^p \from L^p_{\nu} \to L^p_{\nu}$ be defined by
\[(U_z^pf)(w) = f(\phi_z(w))\frac{(1 - \abs{z}^2)^{\frac{\nu+n+1}{p}}}{(1 - \sp{w}{z})^{\frac{2(\nu+n+1)}{p}}}.\]
Using the standard identities mentioned in Section \ref{Notation}, one obtains that $U_z^p$ is a surjective isometry with $(U_z^p)^2 = I$. In particular, $(U_z^q)^*$ is also an isometry. Moreover, it holds $U_z^p(A^p_{\nu}) \subseteq A^p_{\nu}$. However, note that $(U_z^q)^*(A^p_{\nu}) \not\subseteq A^p_{\nu}$ in general so that we have to distinguish between $(U_z^q)^*|_{A^p_{\nu}}$ and $(U_z^q|_{A^q_{\nu}})^* = P_{\nu}(U_z^q)^*|_{A^p_{\nu}}$.

If $A \in \Tf_{p,\nu}$ and $(z_{\gamma})$ is a net in $\B^n$ converging to $x \in \Mf$, the maximal ideal space of $\BUC(\B^n)$ considered as a compactification of $\B^n$ (see \cite[Section 4]{MiSuWi} for a discussion), then $U_{z_{\gamma}}^pA(U_{z_{\gamma}}^q|_{A^q_{\nu}})^*$ converges $*$-strongly in $\Lc(A^p_{\nu})$ (see \cite[Proposition 4.11]{MiSuWi}) and the limit is denoted by $A_x$. If $x$ is located at the boundary $\Mf \setminus \B^n$, we will call the operator $A_x$ a limit operator of $A$, which is in accordance with the sequence space case (\cite{HaLiSe, Lindner, LiSe, RaRoSi, RaRoSi2, Seidel, SpaWi} and the references therein). Note that $\Kc(A^p_{\nu}) \subset \Tf_{p,\nu}$ and $K_x = 0$ for all $x \in \Mf \setminus \B^n$ and $K \in \Kc(A^p_{\nu})$ by \cite[Proposition 4.12, Theorem 5.5]{MiSuWi}.

The set of all limit operators $\set{A_x : x \in \Mf \setminus \B^n}$ is sometimes called the operator spectrum of $A$ because it shares some properties with the usual spectrum, e.g.~some kind of compactness (see Proposition \ref{prop7} below). Note that the operator $A_x$ does not depend on the net $(z_{\gamma})$ converging to $x \in \Mf$ (but of course on the element $x \in \Mf$). Let $b_z \from \Omega \to \C$ be given by
\[b_z(w) := \frac{(1 - \sp{z}{w})^{(\nu+n+1)(1/q-1/p)}}{(1 - \sp{w}{z})^{(\nu+n+1)(1/q-1/p)}}.\]
Then $T_{b_z} = (U_z^q|_{A^q_{\nu}})^*U_z^p|_{A^p_{\nu}}$ is invertible with $T_{b_z}^{-1} = U_z^p(U_z^q|_{A^q_{\nu}})^*$ for all $z \in \B^n$. Moreover, as $z_{\gamma} \to x$ the net $(T_{b_{z_{\gamma}}})$ converges $*$-strongly to another Toeplitz operator, which is denoted by $T_{b_x}$. $T_{b_x}$ is again invertible and $T_{b_{z_{\gamma}}}^{-1} \to T_{b_x}^{-1}$ (see \cite[Lemma 4.10]{MiSuWi}). As we will need it frequently, let us fix the strong continuity in a proposition.

\begin{prop} \label{prop7}
For all $A \in \Tf_{p,\nu}$ the two maps $A_{\bullet} \from \Mf \to \Lc(A^p_{\nu})$, $x \mapsto A_x$ and $T_{b_{\bullet}} \from \Mf \to \Lc(A^p_{\nu})$, $x \mapsto T_{b_x}$ are bounded and continuous w.r.t.~the strong operator topology. In particular, the two sets $\set{A_xT_{b_x} : x \in \Mf}$ and $\set{A_xT_{b_x} : x \in \Mf \setminus \B^n}$ are strongly compact.
\end{prop}

\begin{proof}
This follows directly from \cite[Proposition 4.11]{MiSuWi}.
\end{proof}

So the aim now is to shift a Toeplitz operator $A$ to the boundary to obtain limit operators $A_x$ and then retrieve information about $A$ via Proposition \ref{prop3}. Here is our first step:

\begin{prop} \label{prop4}
Let $p \leq 2$, $\alpha = (\frac{2}{p} - 1)(n+1) + \frac{2\nu}{p}$, $A \in \Tf_{p,\nu}$ and let $(z_{\gamma})$ be a net in $\B^n$ converging to $x \in \Mf \setminus \B^n$ such that $A_x$ is invertible. Then for every real-valued $\xi \in L^{\infty}(\B^n)$ with compact support there is a $\gamma_0$ such that for all $\gamma \geq \gamma_0$ there are operators $B_\gamma,C_{\gamma} \in \Lc(L^p_{\nu})$ with $\norm{B_\gamma}, \norm{C_{\gamma}} \leq 2\left(\norm{A_x^{-1}}\norm{P_{\alpha}} + \norm{Q_{\alpha}}\right)$ and
\[B_\gamma\hat{A}M_{\xi \circ \phi_{z_{\gamma}}} = M_{\xi \circ \phi_{z_{\gamma}}} = M_{\xi \circ \phi_{z_{\gamma}}}\hat{A}C_\gamma.\]
\end{prop}

\begin{proof}
First note that $p \leq 2$ implies $\alpha \geq \nu$ and hence $(\alpha,\nu,p)$ is admissible (cf.~Corollary \ref{BoundedProjectionsCor}). Setting $h(w,z) := 1 - \sp{w}{z}$ and $g := n+1$ and using the standard transformation identities, we observe
\begin{align*}
(U_z^pM_{\xi}U_z^pf)(w) &= \frac{h(z,z)^{\frac{\nu+g}{p}}}{h(w,z)^{\frac{2(\nu+g)}{p}}}(M_{\xi}U_z^pf)(\phi_z(w))\\
&= \frac{h(z,z)^{\frac{\nu+g}{p}}}{h(w,z)^{\frac{2(\nu+g)}{p}}}\xi(\phi_z(w))(U_z^pf)(\phi_z(w))\\
&= \xi(\phi_z(w))\frac{h(z,z)^{\frac{\nu+g}{p}}}{h(w,z)^{\frac{2(\nu+g)}{p}}}\frac{h(z,z)^{\frac{\nu+g}{p}}}{h(\phi_z(w),z)^{\frac{2(\nu+g)}{p}}}f(w)\\
&= \xi(\phi_z(w))\frac{h(z,z)^{\frac{2(\nu+g)}{p}}}{h(w,z)^{\frac{2(\nu+g)}{p}}}\frac{h(w,z)^{\frac{2(\nu+g)}{p}}}{h(z,z)^{\frac{2(\nu+g)}{p}}}f(w)\\
&= \xi(\phi_z(w))f(w)
\end{align*}
so that
\begin{align} \label{eq2}
U_z^pM_{\xi}U_z^p = M_{\xi \circ \phi_z}.
\end{align}

The special value we chose for $\alpha$ also ensures that $P_{\alpha}U_z^p = U_z^pP_{\alpha}$ for all $z \in \B^n$. Indeed,
\begin{align} \label{eq3}
(P_{\alpha}U_z^pf)(x) &= \int_{\B^n} f(\phi_z(w))\frac{h(z,z)^{\frac{\nu+g}{p}}}{h(w,z)^{\frac{2(\nu+g)}{p}}} h(x,w)^{-\alpha-g}\, \mathrm{d}v_{\alpha}(w)\notag\\
&= \int_{\B^n} f(y)\frac{h(z,z)^{\frac{\nu+g}{p}}}{h(\phi_z(y),z)^{\frac{2(\nu+g)}{p}}} h(x,\phi_z(y))^{-\alpha-g} \frac{h(z,z)^{\alpha+g}}{\abs{h(y,z)}^{2(\alpha+g)}} \, \mathrm{d}v_{\alpha}(y)\notag\\
&= \int_{\B^n} f(y)\frac{h(z,z)^{\frac{\nu+g}{p}}h(y,z)^{\frac{2(\nu+g)}{p}}}{h(z,z)^{\frac{2(\nu+g)}{p}}} \frac{h(x,z)^{-\alpha-g}h(\phi_z(x),y)^{-\alpha-g}}{h(z,y)^{-\alpha-g}} \frac{h(z,z)^{\alpha+g}}{\abs{h(y,z)}^{2(\alpha+g)}} \, \mathrm{d}v_{\alpha}(y)\notag\\
&= \frac{h(z,z)^{-\frac{\nu+g}{p}+\alpha+g}}{h(x,z)^{\alpha+g}}\int_{\B^n} f(y)h(y,z)^{\frac{2(\nu+g)}{p}-\alpha-g} h(\phi_z(x),y)^{-\alpha-g} \, \mathrm{d}v_{\alpha}(y)\notag\\
&= \frac{h(z,z)^{\frac{\nu+g}{p}}}{h(x,z)^{\frac{2(\nu+g)}{p}}}\int_{\B^n} f(y)h(\phi_z(x),y)^{-\alpha-g} \, \mathrm{d}v_{\alpha}(y)\notag\\
&= (U_z^pP_{\alpha}f)(x).
\end{align}

Let $B_R := \set{z \in \B^n : \abs{z} \leq R}$, where $R < 1$ is chosen sufficiently large such that $\supp \xi \subseteq B_R$. By Proposition \ref{prop7}, $U_{z_{\gamma}}^pAU_{z_{\gamma}}^p = U_{z_{\gamma}}^pA(U_{z_{\gamma}}^q|_{A^q})^*(U_{z_{\gamma}}^q|_{A^q})^*U_{z_{\gamma}}^p = A_{z_{\gamma}}T_{b_{\gamma}}$ converges $*$-strongly to $A_xT_{b_x}$. Moreover, the operator $P_{\alpha}M_{\chi_{B_R}}$ is compact by Proposition \ref{prop2.5}. Combining these facts and using Equation \eqref{eq3}, we get
\[\norm{\left(U_{z_{\gamma}}^p(AP_{\alpha} + Q_{\alpha})U_{z_{\gamma}}^p - (A_xT_{b_x}P_{\alpha} + Q_{\alpha})\right)M_{\chi_{B_R}}} = \norm{\left(U_{z_{\gamma}}^pAU_{z_{\gamma}}^p - A_xT_{b_x}\right)P_{\alpha}M_{\chi_{B_R}}}  \to 0\]
as $z_{\gamma} \to x$. $A_xT_{b_x}P_{\alpha} + Q_{\alpha}$ is invertible with
\[(A_xT_{b_x}P_{\alpha} + Q_{\alpha})^{-1} = T_{b_x}^{-1}A_x^{-1}P_{\alpha} + Q_{\alpha}.\]
This implies that there exists a $\gamma_0$ such that
\[R_{\gamma} := (A_xT_{b_x}P_{\alpha} + Q_{\alpha})^{-1}\left(U_{z_{\gamma}}^p(AP_{\alpha} + Q_{\alpha})U_{z_{\gamma}}^p - (A_xT_{b_x}P_{\alpha} + Q_{\alpha})\right)M_{\chi_{B_R}}\]
satisfies $\norm{R_{\gamma}} < \frac{1}{2}$ for all $\gamma \geq \gamma_0$. In particular, $I + R_{\gamma} \in \Lc(L^p_{\nu})$ is invertible for all $\gamma \geq \gamma_0$. 

We then have
\[U_{z_{\gamma}}^p(AP_{\alpha} + Q_{\alpha})U_{z_{\gamma}}^pM_{\chi_{B_R}} = (A_xT_{b_x}P_{\alpha} + Q_{\alpha})M_{\chi_{B_R}} + (A_xT_{b_x}P_{\alpha} + Q_{\alpha})R_{\gamma}\]
and therefore
\[(A_xT_{b_x}P_{\alpha} + Q_{\alpha})^{-1}U_{z_{\gamma}}^p(AP_{\alpha} + Q_{\alpha})U_{z_{\gamma}}^pM_{\xi} = M_{\xi} + R_{\gamma}M_{\xi} = (I + R_{\gamma})M_{\xi},\]
which implies
\[(I + R_{\gamma})^{-1}(A_xT_{b_x}P_{\alpha} + Q_{\alpha})^{-1}U_{z_{\gamma}}^p(AP_{\alpha} + Q_{\alpha})U_{z_{\gamma}}^pM_{\xi} = M_{\xi}.\]
Applying $U_{z_{\gamma}}^p$ from both sides and using \eqref{eq2} yields
\[U_{z_{\gamma}}^p(I + R_{\gamma})^{-1}(A_xT_{b_x}P_{\alpha} + Q_{\alpha})^{-1}U_{z_{\gamma}}^p(AP_{\alpha} + Q_{\alpha})M_{\xi \circ \phi_{z_{\gamma}}} = M_{\xi \circ \phi_{z_{\gamma}}}\]
and the first assertion follows. For the second assertion note that $M_{\chi_{B_R}}P_{\alpha}$ is compact as well (see Proposition \ref{prop2.5}). Thus
\[\norm{M_{\chi_{B_R}}\left(U_{z_{\gamma}}^p(AP_{\alpha} + Q_{\alpha})U_{z_{\gamma}}^p - (A_xT_{b_x}P_{\alpha} + Q_{\alpha})\right)} = \norm{M_{\chi_{B_R}}P_{\alpha}\left(U_{z_{\gamma}}^pAU_{z_{\gamma}}^p - A_xT_{b_x}\right)P_{\alpha}} \to 0\]
and we obtain
\[M_{\xi \circ \phi_{z_{\gamma}}}(AP_{\alpha} + Q_{\alpha})U_{z_{\gamma}}^p(A_xT_{b_x}P_{\alpha} + Q_{\alpha})^{-1}(I + S_{\gamma})^{-1}U_{z_{\gamma}}^p = M_{\xi \circ \phi_{z_{\gamma}}}\]
for sufficiently large $\gamma$ and
\[S_{\gamma} := M_{\chi_{B_R}}\left(U_{z_{\gamma}}^p(AP_{\alpha} + Q_{\alpha})U_{z_{\gamma}}^p - (A_xT_{b_x}P_{\alpha} + Q_{\alpha})\right)(A_xT_{b_x}P_{\alpha} + Q_{\alpha})^{-1}.\qedhere\]
\end{proof}

Combining Proposition \ref{prop4} with Proposition \ref{prop3}, we obtain the following theorem.

\begin{thm} \label{thm1}
Let $A \in \Tf_{p,\nu}$. If $A_x$ is invertible for every $x \in \Mf \setminus \B^n$ and $\sup\limits_{x \in \Mf \setminus \B^n} \norm{A_x^{-1}} < \infty$, then $A$ is Fredholm.
\end{thm}

\begin{proof}
W.l.o.g.~we may assume that $p \leq 2$ because otherwise we can just pass to the adjoint, noting that $(A^*)_x = (A_x)^*$ for all $x \in \Mf$.

Let $\psi_{j,t}$ be the functions defined above and assume that $A$ is not Fredholm. It is clear that
\[[\hat{A},P_{\alpha}] = (AP_{\alpha} + Q_{\alpha})P_{\alpha} - P_{\alpha}(AP_{\alpha} + Q_{\alpha}) = AP_{\alpha} - P_{\alpha}AP_{\alpha} = 0\]
as $A \in \Lc(A^p_{\nu})$. Thus by Proposition \ref{prop3}, there is a $t \in (0,1)$ and a strictly increasing sequence $(j_m)_{m \in \N}$ such that
\[B\hat{A}M_{\psi_{j_m,t}} \neq M_{\psi_{j_m,t}} \quad \text{or} \quad M_{\psi_{j_m,t}}\hat {A}B \neq M_{\psi_{j_m,t}}\]
for all $m \in \N$ and all $B \in \Lc(L^p_{\nu})$ with $\norm{B} \leq 2\left(\sup\limits_{x \in \Mf \setminus \B^n} \norm{A_x^{-1}}\norm{P_{\alpha}} + \norm{Q_{\alpha}}\right)$. Taking a suitable subsequence if necessary, we may assume w.l.o.g.~that
\[B\hat{A}M_{\psi_{j_m,t}} \neq M_{\psi_{j_m,t}}\]
for all $m \in \N$ (the other case is exactly the same). By Lemma \ref{lem2}, there is a constant $C$ such that $\diam_{\beta} \supp \psi_{j,t} \leq C$ for all $j \in \N$. We may thus choose a radius $R$ and a sequence of midpoints $(w_m)_{m \in \N}$ with $w_m \to \partial\B^n$ such that
\[\supp \psi_{j_m,t} \subseteq D(w_m,R) = \set{z \in \B^n : \beta(w_m,z) < R}.\]
As $\Mf$ is compact, there exists a subnet $(w_{m_\gamma})$ of $(w_m)$ such that $w_{m_\gamma} \to x$ for some $x \in \Mf \setminus \B^n$. Moreover, choosing $\xi = \chi_{D(0,R)}$ in Proposition \ref{prop4}, we obtain a $\gamma_0$ such that for all $\gamma \geq \gamma_0$ there is an operator $B_\gamma \in \Lc(L^p_{\nu})$ with $\norm{B_\gamma} \leq 2\left(\norm{A_x^{-1}}\norm{P_{\alpha}} + \norm{Q_{\alpha}}\right)$ and
\[B_\gamma\hat{A}M_{\chi_{D(w_{m_\gamma},R)}} = B_\gamma\hat{A}M_{\chi_{D(0,R)} \circ \phi_{w_{m_{\gamma}}}} = M_{\chi_{D(0,R)} \circ \phi_{w_{m_{\gamma}}}} = M_{\chi_{D(w_{m_\gamma},R)}}.\]
Multiplying with $M_{\psi_{j_{m_\gamma,t}}}$ from the right yields
\[B_\gamma\hat{A}M_{\psi_{j_{m_\gamma,t}}} = M_{\psi_{j_{m_\gamma,t}}}\]
for all $\gamma \geq \gamma_0$. This is clearly a contradiction.
\end{proof}

In the next theorem we show that the converse of Theorem \ref{thm1} is true as well. In fact, the converse is not limited to Toeplitz operators.

\begin{thm} \label{thm2}
Let $A \in \Lc(A^p_{\nu})$ be Fredholm and let $(z_{\gamma})$ be a net in $\B^n$ that tends to $x \in \Mf \setminus \B^n$ such that $U_{z_{\gamma}}^pA(U_{z_{\gamma}}^q|_{A^q_{\nu}})^*$ converges $*$-strongly to $A_x \in \Lc(A^p_{\nu})$. Then $A_x$ is invertible and $\norm{A_x^{-1}} \leq \norm{P_{\nu}}\norm{(A + \Kc(A^p_{\nu}))^{-1}}$. Moreover, if $B$ is a Fredholm regularizer of $A$, $U_{z_{\gamma}}^pB(U_{z_{\gamma}}^q|_{A^q_{\nu}})^*$ converges $*$-strongly to $T_{b_x}^{-1}A_x^{-1}T_{b_x}^{-1}$ as $z_{\gamma} \to x$.
\end{thm}

\begin{proof}
If $B$ is a Fredholm regularizer of $A$, then $AB-I$ and $BA-I$ are compact and hence $U_{z_{\gamma}}^p(AB-I)(U_{z_{\gamma}}^q|_{A^q_{\nu}})^* \to 0$ and $U_{z_{\gamma}}^p(BA-I)(U_{z_{\gamma}}^q|_{A^q_{\nu}})^* \to 0$ $*$-strongly as $z_{\gamma} \to x$ (see \cite[Proposition 4.12, Theorem 5.5]{MiSuWi}). Moreover,
\begin{align*}
\norm{f} &= \norm{T_{b_{z_{\gamma}}}T_{b_{z_{\gamma}}}^{-1}f} \leq \norm{P_{\nu}} \norm{U_{z_{\gamma}}^p(U_{z_{\gamma}}^q|_{A^q_{\nu}})^*f}\\
&\leq \norm{P_{\nu}}\norm{U_{z_{\gamma}}^pBA(U_{z_{\gamma}}^q|_{A^q_{\nu}})^*f} + \norm{P_{\nu}}\norm{U_{z_{\gamma}}^p(I-BA)(U_{z_{\gamma}}^q|_{A^q_{\nu}})^*f}\\
&\leq \norm{P_{\nu}}\norm{B}\norm{U_{z_{\gamma}}^pA(U_{z_{\gamma}}^q|_{A^q_{\nu}})^*f} + \norm{P_{\nu}}\norm{U_{z_{\gamma}}^p(I-BA)(U_{z_{\gamma}}^q|_{A^q_{\nu}})^*f}
\end{align*}
for every $f \in A^p_{\nu}$, using that $U_{z_{\gamma}}^p$ is an isometry. Taking the limit $z_{\gamma} \to x$, we obtain $\norm{f} \leq \norm{P_{\nu}}\norm{B}\norm{A_xf}$ for every $f \in A^p_{\nu}$. This implies that $A_x$ is injective and has a closed range. By the dual argument, we also obtain $\norm{f} \leq \norm{P_{\nu}}\norm{B^*}\norm{A_x^*f}$ for every $f \in A^q_{\nu}$, which implies that $A_x$ is surjective, hence invertible. Moreover, it shows that $\norm{A_x^{-1}} \leq \norm{P_{\nu}}\norm{B}$. As this is true for every regularizer $B$, we obtain $\norm{A_x^{-1}} \leq \norm{P_{\nu}}\norm{(A + \Kc(A^p_{\nu}))^{-1}}$.

Moreover, using $(U_{z_{\gamma}}^q|_{A^q_{\nu}})^*T_{b_{z_{\gamma}}}U_{z_{\gamma}}^p = I$, we have
\begin{align*}
U_{z_{\gamma}}^pB(U_{z_{\gamma}}^q|_{A^q_{\nu}})^* - T_{b_x}^{-1}A_x^{-1}T_{b_x}^{-1} &= U_{z_{\gamma}}^pB(U_{z_{\gamma}}^q|_{A^q_{\nu}})^*T_{b_{z_{\gamma}}}(A_x - U_{z_{\gamma}}^pA(U_{z_{\gamma}}^q|_{A^q_{\nu}})^*)A_x^{-1}T_{b_{z_{\gamma}}}^{-1}\\
&\quad + U_{z_{\gamma}}^p(BA-I)(U_{z_{\gamma}}^q|_{A^q_{\nu}})^*A_x^{-1}T_{b_{z_{\gamma}}}^{-1}\\
&\quad + T_{b_{z_{\gamma}}}^{-1}A_x^{-1}(T_{b_{z_{\gamma}}}^{-1} - T_{b_{x}}^{-1}) + (T_{b_{z_{\gamma}}}^{-1} - T_{b_{x}}^{-1})A_x^{-1}T_{b_x}^{-1}
\end{align*}
and all terms on the right-hand side tend $*$-strongly to $0$ as $z_{\gamma} \to x$.
\end{proof}

In particular, we have shown that a Toeplitz operator is Fredholm if and only if all of its limit operators are invertible and their inverses are uniformly bounded. We will state this result in a separate theorem below. But let us first argue why the condition on uniform boundedness is actually redundant. The argument is very similar to the sequence space case, cf.~\cite{LiSe}.

Let $r_t := \sup\limits_{j \in \N} \diam_{\beta} \supp \varphi_{j,t}$, where $\varphi_{j,t}$ is defined as usual. By Lemma \ref{lem2}, $r_t$ is finite for every $t \in (0,1)$. Now, for every $t \in (0,1)$, $A \in \Lc(L^p_{\nu})$ and every Borel set $F \subseteq \B^n$ we define
\[\nu_t(A|_F) := \inf\set{\norm{Af} : f \in L^p_{\nu}, \norm{f} = 1, \supp f \subseteq D(w,r_t) \cap F \text{ for some } w \in \B^n}\]
and
\[\nu(A|_F) := \inf\set{\norm{Af} : f \in L^p_{\nu}, \norm{f} = 1, \supp f \subseteq F}.\]
Moreover, $\nu(A) := \nu(A|_{\B^n})$.

\begin{prop} \label{prop6}
For all $A,B \in \Lc(L^p_{\nu})$ and all Borel sets $F \subseteq \B^n$ it holds $\abs{\nu(A|_F) - \nu(B|_F)} \leq \norm{(A - B)M_{\chi_F}}$. The same statement also holds if we replace $\nu$ by $\nu_t$ for some $t \in (0,1)$.
\end{prop}

\begin{proof}
We only prove the first claim, but the same proof also works for the second claim. Let $\epsilon > 0$. Choose $f \in L^p_{\nu}$ with $\norm{f} = 1$, $\supp f \subseteq F$ and $\norm{Bf} \leq \nu(B|_F) + \epsilon$. This implies
\[\nu(A|_F) - \nu(B|_F) - \epsilon \leq \nu(A|_F) - \norm{Bf} \leq \norm{Af} - \norm{Bf} \leq \norm{(A - B)f} \leq \norm{(A - B)M_{\chi_F}}.\]
Similarly, $\nu(B|_F) - \nu(A|_F) - \epsilon \leq \norm{(A - B)M_{\chi_F}}$. Since $\epsilon$ was arbitrary, the assertion follows. 
\end{proof}

For $p \leq 2$ and $\alpha = (\frac{2}{p} - 1)(n+1) + \frac{2\nu}{p}$ we will use the (abuse of) notation $\hat{A}_x := A_xT_{b_x}P_{\alpha} + Q_{\alpha}$.

\begin{prop} \label{prop5}
Let $p \leq 2$, $\alpha = (\frac{2}{p} - 1)(n+1) + \frac{2\nu}{p}$ and $A \in \Tf_{p,\nu}$. Then for every $\epsilon > 0$ there exists a $t \in (0,1)$ such that for every Borel set $F \subseteq \B^n$ and every operator $B \in \{\hat{A}\} \cup \set{\hat{A}_x : x \in \Mf \setminus \B^n}$ it holds
\begin{equation} \label{eq_prop5}
\nu(B|_F) \leq \nu_t(B|_F) \leq \nu(B|_F) + \epsilon.
\end{equation}
\end{prop}

\begin{proof}
The first inequality is clear by definition. For the second inequality we start with a few simple observations. By Theorem \ref{thm0}, $\hat{A}$ is band-dominated. Therefore there is a sequence of band operators $(\hat{A}_n)_{n \in \N}$ that converges to $\hat{A}$. Choose $n$ sufficiently large such that $\|\hat{A}-\hat{A}_n\| < \frac{\epsilon}{4}$ and denote the band width of $\hat{A}_n$ by $\omega$. Let $x \in \Mf$ and choose a net $(z_{\gamma})$ that converges to $x$. Then $(U_{z_{\gamma}}\hat{A}_nU_{z_{\gamma}})$ is a bounded net and hence there is a subnet of $(z_{\gamma})$, again denoted by $(z_{\gamma})$ such that $(U_{z_{\gamma}}\hat{A}_nU_{z_{\gamma}})$ converges in the weak operator topology as $z_{\gamma} \to x$. Let us denote this limit by $(\hat{A}_x)_n$. As $(U_{z_{\gamma}}\hat{A}U_{z_{\gamma}}) = (U_{z_{\gamma}}AU_{z_{\gamma}})P_{\alpha} + Q_{\alpha}$ converges to $A_xT_{b_x}P_{\alpha} + Q_{\alpha} = \hat{A}_x$ in the strong operator topology (see Proposition \ref{prop7}), we obtain that $(U_{z_{\gamma}}(\hat{A}-\hat{A}_n)U_{z_{\gamma}})$ converges to $\hat{A}_x - (\hat{A}_x)_n$ in the weak operator topology. This implies
\[\|\hat{A}_x - (\hat{A}_x)_n\| \leq \sup\limits_{\gamma} \|(U_{z_{\gamma}}(\hat{A}-\hat{A}_n)U_{z_{\gamma}})\| = \|\hat{A}-\hat{A}_n\| < \frac{\epsilon}{4}.\]
Let $f,g \in L^{\infty}(\B^n)$ with $\dist_{\beta}(\supp f,\supp g) > \omega$. Then equation \eqref{eq2} implies
\[M_f(U_{z_{\gamma}}\hat{A}_nU_{z_{\gamma}})M_g = U_{z_{\gamma}}M_{f \circ \phi_{z_{\gamma}}}\hat{A}_nM_{g \circ \phi_{z_{\gamma}}}U_{z_{\gamma}} = 0\]
because $\dist_{\beta}(\supp f \circ \phi_{z_{\gamma}},\supp g \circ \phi_{z_{\gamma}}) = \dist_{\beta}(\supp f,\supp g) > \omega$. Hence all elements in the net $(U_{z_{\gamma}}\hat{A}_nU_{z_{\gamma}})$ have the same band width $\omega$. As $M_f(U_{z_{\gamma}}\hat{A}_nU_{z_{\gamma}})M_g$ converges to $M_f(\hat{A}_x)_nM_g$ in weak operator topology, $(\hat{A}_x)_n$ is also a band operator of band width at most $\omega$.

The strategy now is to prove that there exists a $t \in (0,1)$ such that for all $F \subseteq \B^n$ and all operators $B \in \{\hat{A}_n\} \cup \set{(\hat{A}_x)_n : x \in \Mf \setminus \B^n}$ it holds
\[\nu_t(B|_F) \leq \nu(B|_F) + \frac{\epsilon}{2}\]
and then use Proposition \ref{prop6}. Indeed, suppose that the above is true. Then Proposition \ref{prop6} implies
\[|\nu(\hat{A}|_F) - \nu(\hat{A}_n|_F)| \leq \|\hat{A} - \hat{A}_n\| < \frac{\epsilon}{4} \quad \text{and} \quad |\nu(\hat{A}_x|_F) - \nu((\hat{A}_x)_n|_F)| \leq \|\hat{A}_x - (\hat{A}_x)_n\| < \frac{\epsilon}{4}\]
and
\[|\nu_t(\hat{A}|_F) - \nu_t(\hat{A}_n|_F)| \leq \|\hat{A} - \hat{A}_n\| < \frac{\epsilon}{4} \quad \text{and} \quad \|\nu_t(\hat{A}_x|_F) - \nu_t((\hat{A}_x)_n|_F)| \leq \|\hat{A}_x - (\hat{A}_x)_n\| < \frac{\epsilon}{4}\]
for all $t \in (0,1)$ and the proposition follows.

Choose $f \in L^p_{\nu}$ with $\norm{f} = 1$ and $\supp f \subseteq F$ such that
\[\|Bf\| \leq \nu(B|_F) + \frac{\epsilon}{4}.\]
Let $\varphi_{j,t}$ and $\psi_{j,t}$ be defined as usual. Then by Minkowski's inequality in $\ell^p(\N)$, we obtain
\begin{align*}
\left(\sum\limits_{j = 1}^{\infty} \norm{BM_{\varphi_{j,t}^{1/p}}f}^p\right)^{1/p} &= \left(\sum\limits_{j = 1}^{\infty} \norm{BM_{\varphi_{j,t}^{1/p}}M_{\psi_{j,t}}f}^p\right)^{1/p}\\
&\leq \left(\sum\limits_{j = 1}^{\infty} \norm{M_{\varphi_{j,t}^{1/p}}Bf}^p\right)^{1/p} + \left(\sum\limits_{j = 1}^{\infty} \norm{M_{\varphi_{j,t}^{1/p}}BM_{1-\psi_{j,t}}f}^p\right)^{1/p}\\
&\qquad + \left(\sum\limits_{j = 1}^{\infty} \norm{[B,M_{\varphi_{j,t}^{1/p}}]M_{\psi_{j,t}}f}^p\right)^{1/p}.
\end{align*}
The first term is just $\|Bf\|$ (recall that $\sum\limits_{j = 1}^{\infty} \abs{\varphi_{j,t}(z)} = 1$ for all $z \in \B^n$, $t \in (0,1)$). The second term vanishes for $\dist_{\beta}(\supp \varphi_{j,t},\supp(1 - \psi_{j,t})) > \omega$ as $B$ has band width $\omega$. The third term can be estimated as
\begin{align} \label{eq_prop5_2}
\left(\sum\limits_{j = 1}^{\infty} \norm{[B,M_{\varphi_{j,t}^{1/p}}]M_{\psi_{j,t}}f}^p\right)^{1/p} &\leq \sup\limits_{j \in \N} \norm{[B,M_{\varphi_{j,t}^{1/p}}]}\left(\sum\limits_{j = 1}^{\infty} \norm{M_{\psi_{j,t}}f}^p\right)^{1/p}\notag\\
&\leq N^{1/p}\sup\limits_{j \in \N} \norm{[B,M_{\varphi_{j,t}^{1/p}}]}
\end{align}
by Proposition \ref{AuxiliaryProp2}. Observe that the functions $\varphi_{j,t}^{1/p}$ satisfy the assumptions in Lemma \ref{lem3}. Indeed, let $U,V \subset [0,1]$ with $\dist(U,V) > 0$ and $w_{j,t} \in (\varphi_{j,t}^{1/p})^{-1}(U)$, $z_{j,t} \in (\varphi_{j,t}^{1/p})^{-1}(V)$. Clearly, we have $\dist(U^p,V^p) > 0$ as well and therefore
\[\beta(z_{j,t},w_{j,t}) \geq \frac{1}{6Nt}\abs{\varphi_{j,t}(z_{j,t}) - \varphi_{j,t}(w_{j,t})} \geq \frac{1}{6Nt}\dist(U^p,V^p) \to \infty\]
as $t \to 0$. Lemma \ref{lem3} thus implies that for every $\delta > 0$ there is a $t > 0$ such that
\[\left(\sum\limits_{j = 1}^{\infty} \norm{[B,M_{\varphi_{j,t}^{1/p}}]M_{\psi_{j,t}}f}^p\right)^{1/p} \leq \delta\norm{B}.\]
As $\norm{B} \leq \|\hat{A}\| + \frac{\epsilon}{4}$ for all $B \in \{\hat{A}_n\} \cup \set{(\hat{A}_x)_n : x \in \Mf \setminus \B^n}$ by the above, we may choose $\delta > 0$ such that $\delta\norm{B} \leq \frac{\epsilon}{4}$ for all $B$. Therefore
\[\left(\sum\limits_{j = 1}^{\infty} \norm{BM_{\varphi_{j,t}^{1/p}}f}^p\right)^{1/p} \leq  \|Bf\| + \frac{\epsilon}{4} \leq \nu(B|_F) + \frac{\epsilon}{2} = \left(\nu(B|_F) + \frac{\epsilon}{2}\right)\left(\sum\limits_{j = 1}^{\infty} \norm{M_{\varphi_{j,t}^{1/p}}f}^p\right)^{1/p}.\]
This implies, in particular, that there exists a $j \in \N$ such that
\[\norm{BM_{\varphi_{j,t}^{1/p}}f} \leq (\nu(B|_F) + \frac{\epsilon}{2})\norm{M_{\varphi_{j,t}^{1/p}}f}\]
for sufficiently small $t$. Since $\supp \left(M_{\varphi_{j,t}^{1/p}}f\right) \subseteq \supp \varphi_{j,t} \subseteq D(w,r_t)$ for some $w \in \B^n$ by definition, this implies $\nu_t(B|_F) \leq \nu(B|_F) + \frac{\epsilon}{2}$ for all $B \in \{\hat{A}_n\} \cup \set{(\hat{A}_x)_n : x \in \Mf \setminus \B^n}$. As $t$ is chosen independently of $F$ (as it is chosen independently of $f$) and $B$, the assertion follows.
\end{proof}

The next lemma shows some kind of shift invariance of the operator spectrum. This will allow us to shift functions to the right place in the subsequent lemma.

\begin{lem} \label{lem5}
Let $p \leq 2$, $\alpha = (\frac{2}{p} - 1)(n+1) + \frac{2\nu}{p}$, $A \in \Tf_{p,\nu}$ and $f \in L^p_{\nu}$ with $\supp f \subseteq D(w,r)$ for some $w \in \B^n$ and $r > 0$. Then for every $x \in \Mf \setminus \B^n$ there exist $y \in \Mf \setminus \B^n$ and $g \in L^p_{\nu}$ with $\norm{g} = \norm{f}$ and $\supp g \subseteq D(0,r)$ such that $\|\hat{A}_xf\| = \|\hat{A}_yg\|$. Moreover, $\nu(\hat{A}_y|_{D(0,r + \beta(0,w))}) \leq \nu(\hat{A}_x|_{D(0,r)})$.
\end{lem}

\begin{proof}
A direct computation yields
\[(U_z^pU_w^pf)(\zeta) = (f \circ \phi_w \circ \phi_z)(\zeta) \frac{(1 - \abs{\phi_z(w)}^2)^{(\nu+n+1)/p}}{(1 - \sp{\zeta}{\phi_z(w)})^{2(\nu+n+1)/p}} \left(\frac{1 - \sp{w}{z}}{\abs{1 - \sp{w}{z}}}\right)^{2(\nu+n+1)/p}\]
for all $w,z \in \B^n$. Using Cartan's theorem, i.e.~$\phi_w \circ \phi_z = V \circ \phi_{\phi_z(w)}$ for some unitary map $V \from \C^n \to \C^n$, we get
\[U_z^pU_w^p = \left(\frac{1 - \sp{w}{z}}{\abs{1 - \sp{w}{z}}}\right)^{2(\nu+n+1)/p}U_{\phi_z(w)}^pV_*,\]
where $V_*f := f \circ V$ is a composition operator and, by taking inverses, also
\[U_w^pU_z^p = \left(\frac{1 - \sp{z}{w}}{\abs{1 - \sp{w}{z}}}\right)^{2(\nu+n+1)/p}V_*^{-1}U_{\phi_z(w)}^p.\]
Choose a net $(z_{\gamma})$ that converges to $x \in \Mf \setminus \B^n$. Then
\begin{align*}
U_w^pU_{z_{\gamma}}^pAU_{z_{\gamma}}^pU_w^p|_{A^p_{\nu}} &= \left(\frac{1 - \sp{z_{\gamma}}{w}}{\abs{1 - \sp{w}{z_{\gamma}}}}\right)^{2(\nu+n+1)/p}\left(\frac{1 - \sp{w}{z_{\gamma}}}{\abs{1 - \sp{w}{z_{\gamma}}}}\right)^{2(\nu+n+1)/p}V_*^{-1}U_{\phi_{z_{\gamma}}(w)}^pA\\
&\quad \cdot U_{\phi_{z_{\gamma}}(w)}^pV_*|_{A^p_{\nu}}\\
&= V_*^{-1}U_{\phi_{z_{\gamma}}(w)}^pAU_{\phi_{z_{\gamma}}(w)}^pV_*|_{A^p_{\nu}},
\end{align*}
where $V_*$ of course depends on $w$ and $z_{\gamma}$. As the unitary group of $\C^n$ is compact, we may assume that $V$ converges to some unitary map $\tilde{V}$ and hence $V_*|_{A^p_{\nu}}$ converges strongly to $\tilde{V}_*|_{A^p_{\nu}}$ and $V_*^{-1}|_{A^p_{\nu}}$ converges strongly to $\tilde{V}^{-1}_*|_{A^p_{\nu}}$ as $z_{\gamma} \to x$. Similarly, using Proposition \ref{prop7}, we may assume that
\[U_{\phi_{z_{\gamma}}(w)}^pAU_{\phi_{z_{\gamma}}(w)}^p|_{A^p_{\nu}} = U_{\phi_{z_{\gamma}}(w)}^pA(U_{\phi_{z_{\gamma}}(w)}^q|_{A^q_{\nu}})^*(U_{\phi_{z_{\gamma}}(w)}^q|_{A^q_{\nu}})^*U_{\phi_{z_{\gamma}}(w)}^p|_{A^p_{\nu}} = A_{\phi_{z_{\gamma}}(w)}T_{b_{\phi_{z_{\gamma}}(w)}}\] 
converges strongly to $A_yT_{b_y}$ for some $y \in \Mf$. Since $\phi_{z_{\gamma}}(w) \to \partial\B^n$ as $z_{\gamma} \to \partial\B^n$, it is clear that $y \in \Mf \setminus \B^n$. As the limit of a strongly convergent net is unique and $\lim\limits_{z_{\gamma} \to x} U_{z_{\gamma}}^pAU_{z_{\gamma}}^p = A_xT_{b_x}$, we obtain
\[U_w^pA_xT_{b_x}U_w^p|_{A^p_{\nu}} = \tilde{V}_*^{-1}A_yT_{b_y}\tilde{V}_*|_{A^p_{\nu}}.\]
Since $P_{\alpha}$ commutes with both $U_w^p$ and $\tilde{V}_*$ (the former was shown in the proof of Proposition \ref{prop4}, the latter is immediate from the definition of $P_{\alpha}$ and $\tilde{V}_*$), this implies
\[U_w^p\hat{A}_xU_w^p = \tilde{V}_*^{-1}\hat{A}_y\tilde{V}_*.\]

Now let $f \in L^p_{\nu}$ with $\supp f \subseteq D(w,r)$ for some $w \in \B^n$ and $r > 0$. Set $g := \tilde{V}_*U_w^pf$. Then clearly $\norm{g} = \norm{f}$ and $\|\hat{A}_xf\| = \|\hat{A}_yg\|$. As $D(0,r)$ is invariant under $\tilde{V}$, it remains to show that $\supp U_w^pf \subseteq D(0,r)$. But this is clear since $\phi_w(D(0,r)) = D(w,r)$.

For the second assertion consider $f \in L^p_{\nu}$ with $\supp f \subseteq D(0,r)$. Then $g := \tilde{V}_*U_w^pf$ satisfies $\norm{g} = \norm{f}$, $\|\hat{A}_xf\| = \|\hat{A}_yg\|$ and $\supp g \subseteq \tilde{V}^{-1}(\phi_w^{-1}(D(0,r))) \subseteq D(0,r + \beta(0,w))$ as above.
\end{proof}

In the next lemma we show that the infimum $\inf\set{\nu(\hat{A}_x) : x \in \Mf \setminus \B^n}$ is actually attained by a certain limit operator $\hat{A}_y$. The proof is very much the same as the proofs of \cite[Theorem 3.2]{HaLiSe} and \cite[Theorem 8]{LiSe} and is based on a miraculous procedure invented by Markus Seidel. We therefore refer to \cite{LiSe} for a helpful illustration of the main idea.

\begin{lem} \label{lem6}
Let $p \leq 2$, $\alpha = (\frac{2}{p} - 1)g + \frac{2\nu}{p}$ and $A \in \Tf_{p,\nu}$. Then there exists a $y \in \Mf \setminus \B^n$ with
\[\nu(\hat{A}_y) = \inf\set{\nu(\hat{A}_x) : x \in \Mf \setminus \B^n}.\]
\end{lem}

\begin{proof}
Recall $r_t = \sup\limits_{j \in \N} \diam_{\beta}\supp \varphi_{j,t}$. Proposition \ref{prop5} yields a sequence $(t_k)_{k \in \N_0}$ with $r_{t_{k+1}} > 2r_{t_k}$ and
\[\nu_{t_k}(B|_F) \leq \nu(B|_F) + 2^{-(k+1)}\]
for all $k \in \N_0$, $F \subseteq \B^n$ and $B \in \{\hat{A}\} \cup \set{\hat{A}_x : x \in \Mf \setminus \B^n}$. Choose a sequence $(x_n)_{n \in \N}$ in $\Mf \setminus \B^n$ such that
\[\lim\limits_{n \to \infty} \nu(\hat{A}_{x_n}) = \inf\set{\nu(\hat{A}_x) : x \in \Mf \setminus \B^n}.\]
Now, for every $n \in \N$ there exists an $f_n^0 \in L^p_{\nu}$ with $\|f_n^0\| = 1$, $\supp f_n^0$ contained in some $D(w,r_{t_n})$ and
\[\norm{\hat{A}_{x_n}f_n^0} \leq \nu_{t_n}(\hat{A}_{x_n}) + 2^{-(n+1)} \leq \nu(\hat{A}_{x_n}) + 2^{-n}.\]
By Lemma \ref{lem5}, we can choose a $y_n^0 \in \Mf \setminus \B^n$ and a $g_n^0 \in L^p_{\nu}$ with $\|g_n^0\| = 1$ and $\supp g_n^0 \subseteq D(0,r_{t_n})$ such that
\[\norm{\hat{A}_{y_n^0}g_n^0} = \norm{\hat{A}_{x_n}f_n^0} \leq \nu(\hat{A}_{x_n}) + 2^{-n}.\]
Furthermore, for every $n \in \N$ we can find $f_n^1 \in L^p_{\nu}$ with $\|f_n^1\| = 1$, $\supp f_n^1 \subseteq D(w,r_{t_{n-1}}) \cap D(0,r_{t_n})$ for some $w \in D(0,r_{t_n}+r_{t_{n-1}})$ and
\[\norm{\hat{A}_{y_n^0}f_n^1} \leq \nu_{t_{n-1}}(\hat{A}_{y_n^0}|_{D(0,r_{t_n})}) + 2^{-n} \leq \nu(\hat{A}_{y_n^0}|_{D(0,r_{t_n})}) + 2^{-n+1}.\]
Thus, using Lemma \ref{lem5} again, we can choose $y_n^1 \in \Mf \setminus \B^n$ and $g_n^1 \in L^p_{\nu}$ with $\|g_n^1\| = 1$ and $\supp g_n^1 \subseteq D(0,r_{t_{n-1}})$ such that
\[\norm{\hat{A}_{y_n^1}g_n^1} = \norm{\hat{A}_{y_n^0}f_n^1} \leq \nu(\hat{A}_{y_n^0}|_{D(0,r_{t_n})}) + 2^{-n+1}.\]
In particular,
\[\nu(\hat{A}_{y_n^1}|_{D(0,r_{t_{n-1}})}) \leq \nu(\hat{A}_{y_n^0}|_{D(0,r_{t_n})}) + 2^{-n+1} \leq \nu(\hat{A}_{x_n}) + 2^{-n+1} + 2^{-n}.\]
Repeating this argument, we obtain $y_n^k \in \Mf \setminus \B^n$, $w \in D(0,r_{t_{n-k+1}} + r_{t_{n-k}})$ and $f_n^k \in L^p_{\nu}$ with $\norm{f_n^k} = 1$, $\supp f_n^k \subseteq D(w,r_{t_{n-k}}) \cap D(0,r_{t_{n-k+1}})$ and
\[\norm{\hat{A}_{y_n^{k-1}}f_n^k} \leq \nu_{t_{n-k}}(\hat{A}_{y_n^{k-1}}|_{D(0,r_{t_{n-k+1}})}) + 2^{-n+k-1} \leq \nu(\hat{A}_{y_n^{k-1}}|_{D(0,r_{t_{n-k+1}})}) + 2^{-n+k}\]
for $k = 2, \ldots, n$.
Moreover, we obtain $g_n^k \in L^p_{\nu}$ with $\|g_n^k\| = 1$ and $\supp g_n^k \subseteq D(0,r_{t_{n-k}})$ such that
\[\norm{\hat{A}_{y_n^k}g_n^k} = \norm{\hat{A}_{y_n^{k-1}}f_n^k} \leq \nu(\hat{A}_{y_n^{k-1}}|_{D(0,r_{t_{n-k+1}})}) + 2^{-n+k}.\]
In particular,
\begin{align*}
\nu(\hat{A}_{y_n^k}|_{D(0,r_{t_{n-k}})}) &\leq \nu(\hat{A}_{y_n^{k-1}}|_{D(0,r_{t_{n-k+1}})}) + 2^{-n+k} \leq \ldots \leq \nu(\hat{A}_{y_n^0}|_{D(0,r_{t_n})}) + 2^{-n+k} + \ldots + 2^{-n+1}\\
&\leq \nu(\hat{A}_{x_n}) + 2^{-n+k} + \ldots + 2^{-n+1} + 2^{-n} \leq \nu(\hat{A}_{x_n}) + 2^{-n+k+1}.
\end{align*}
Furthermore, by repeatedly applying the second part of Lemma \ref{lem5} and collecting all the shifts being made during the process above, we get
\begin{align*}
\nu(\hat{A}_{y_n^{n-l}}|_{D(0,r_{t_l})}) &\geq \nu(\hat{A}_{y_n^{n-l+1}}|_{D(0,r_{t_l} + r_{t_l} + r_{t_{l-1}})}) = \nu(\hat{A}_{y_n^{n-l+1}}|_{D(0,2r_{t_l} + r_{t_{l-1}})})\\
&\geq \nu(\hat{A}_{y_n^{n-l+2}}|_{D(0,2r_{t_l} + r_{t_{l-1}} + r_{t_{l-1}} + r_{t_{l-2}})}) = \nu(\hat{A}_{y_n^{n-l+2}}|_{D(0,2r_{t_l} + 2r_{t_{l-1}} + r_{t_{l-2}})})\\
&\geq \ldots \geq \nu(\hat{A}_{y_n^n}|_{D(0,2r_{t_l} + 2r_{t_{l-1}} + 2r_{t_{l-2}} + \ldots + 2r_{t_1} + r_{t_0})}) \geq \nu(\hat{A}_{y_n^n}|_{D(0,4r_{t_l})})
\end{align*}
for fixed $l \leq n$, using $r_{t_{k+1}} > 2r_{t_k}$ for all $k$.

Now set $y_n:= y_n^n$. By Proposition \ref{prop7}, the sequence $(A_{y_n}T_{b_{y_n}})_{n \in \N}$ has a strongly convergent subnet, denoted by $(A_{y_{n_{\gamma}}}T_{b_{y_{n_{\gamma}}}})$, that converges to $A_yT_{b_y}$ for some $y \in \Mf \setminus \B^n$. Now since $\overline{D(0,4r_{t_l})}$ is a compact set, $P_{\alpha}M_{\chi_{D(0,4r_{t_l})}} \in \Lc(L^p_{\nu})$ is a compact operator (see Proposition \ref{prop2.5}) and therefore
\[\norm{(\hat{A}_{y_{n_{\gamma}}} - \hat{A}_y)M_{\chi_{D(0,4r_{t_l})}}} = \norm{(A_{y_{n_{\gamma}}}T_{b_{y_{n_{\gamma}}}} - A_yT_{b_y})P_{\alpha}M_{\chi_{D(0,4r_{t_l})}}} \to 0.\]
This also implies
$\nu(\hat{A}_{y_{n_{\gamma}}}|_{D(0,4r_{t_l})}) \to \nu(\hat{A}_y|_{D(0,4r_{t_l})})$ by Proposition \ref{prop6}. Thus
\begin{align*}
\nu(\hat{A}_y) &\leq \nu(\hat{A}_y|_{D(0,4r_{t_l})}) = \lim\limits_{\gamma} \nu(\hat{A}_{y_{n_{\gamma}}}|_{D(0,4r_{t_l})}) \leq \lim\limits_{\gamma} \nu(\hat{A}_{y_{n_{\gamma}}^{n_{\gamma}-l}}|_{D(0,r_{t_l})}) \leq \lim\limits_{\gamma} \nu(\hat{A}_{x_{n_{\gamma}}}) + 2^{-l+1}\\
&= \inf\set{\nu(\hat{A}_x) : x \in \Mf \setminus \B^n} + 2^{-l+1}.
\end{align*}
Sending $l \to \infty$, we obtain $\nu(\hat{A}_y) = \inf\set{\nu(\hat{A}_x) : x \in \Mf \setminus \B^n}$ as claimed.
\end{proof}

Now we can summarize this section with the main result of this paper.

\begin{mthm} \label{thm4}
Let $A \in \Tf_{p,\nu}$. Then the following are equivalent:
\begin{itemize}
\item[$(i)$] $A$ is Fredholm,
\item[$(ii)$] $A_x$ is invertible and $\norm{A_x^{-1}} \leq \norm{P_{\nu}}\norm{(A + \Kc(A^p_{\nu}))^{-1}}$ for all $x \in \Mf \setminus \B^n$,
\item[$(iii)$] $A_x$ is invertible for all $x \in \Mf \setminus \B^n$ and $\sup\limits_{x \in \Mf \setminus \B^n} \norm{A_x^{-1}} < \infty$,
\item[$(iv)$] $A_x$ is invertible for all $x \in \Mf \setminus \B^n$,
\end{itemize}
\end{mthm}

\begin{proof}
That $(i)$ implies $(ii)$ follows from Proposition \ref{prop7} and Theorem \ref{thm2}, whereas $(ii)$ obviously implies $(iii)$ and $(iii)$ implies $(iv)$. It remains to show that $(iv)$ implies $(i)$. By duality, it suffices to show the case $p \leq 2$. If $A_x$ is invertible, then $(\hat{A}_x)^{-1} = T_{b_x}^{-1}A_x^{-1}P_{\alpha} + Q_{\alpha}$, which imnplies that $(\hat{A}_x)^{-1}$ is also invertible. Now observe that $\nu(B) = \norm{B^{-1}}^{-1} > 0$ whenever an operator $B \neq 0$ is invertible (see e.g.~\cite[Lemma 2.35]{Lindner} for a quick proof). Thus by Lemma \ref{lem6}, $\sup\limits_{x \in \Mf \setminus \B^n} \|\hat{A}_x^{-1}\| < \infty$. Moreover, if $\hat{A}_x$ is invertible, then $T_{b_x}\hat{A}_x^{-1}|_{A^p_{\nu}}$ is the inverse of $A_x$ and so $\sup\limits_{x \in \Mf \setminus \B^n} \norm{A_x^{-1}} < \infty$. Therefore $(iv)$ implies $(iii)$. That $(iii)$ implies $(i)$ is Theorem \ref{thm1}.
\end{proof}

Of course, we get the following corollary for the essential spectrum:

\begin{cor} \label{cor3}
Let $A \in \Tf_{p,\nu}$. Then
\[\spec_{\ess}(A) = \bigcup\limits_{x \in \Mf \setminus \B^n} \spec(A_x).\]
\end{cor}

\section{Norm estimates (unit ball)} \label{Norm_estimates}

In this section we prove a similar result for the essential norm of an operator $A \in \Tf_{p,\nu}$. For the most part this is just a modification of the proofs in Section \ref{Limit_operators}. Recall that $r_t = \sup\limits_{j \in \N} \diam_{\beta}\supp \varphi_{j,t}$. For $t \in (0,1)$, $\alpha \geq \nu$, $A \in \Lc(A^p_{\nu})$ and a Borel set $F \subseteq \B^n$ we define
\[\vertiii{AP_{\alpha}|_F}_t := \sup\set{\norm{AP_{\alpha}f} : f \in L^p_{\nu}, \norm{f} = 1, \supp f \subseteq D(w,r_t) \cap F \text{ for some } w \in \B^n}\]
and
\[\norm{AP_{\alpha}|_F} := \sup\set{\norm{AP_{\alpha}f} : f \in L^p_{\nu}, \norm{f} = 1, \supp f \subseteq F}.\]

\begin{prop} \label{prop9}
Let $p \leq 2$, $\alpha = (\frac{2}{p} - 1)(n+1) + \frac{2\nu}{p}$ and $A \in \Tf_{p,\nu}$. Then for every $\epsilon > 0$ there exists a $t \in (0,1)$ such that for all Borel Sets $F \subseteq \B^n$ and all operators $B$ in the set
\[\{A\} \cup \set{A_xT_{b_x} : x \in \Mf \setminus \B^n}\]
it holds
\[\|BP_{\alpha}|_F\| \geq \vertiiis{BP_{\alpha}|_F}_t \geq \|BP_{\alpha}|_F\| - \epsilon.\]
\end{prop}

\begin{proof}
The first inequality is clear by definition. For the second inequality we proceed as in the proof of Proposition \ref{prop5}. Let $A \in \Tf_{p,\nu}$ and fix $\epsilon > 0$. Observe that $AP_{\alpha} = \hat{A} - Q_{\alpha}$ is band-dominated by Theorem \ref{thm0}. We may thus choose a band operator $A_n$ of band width $\omega$ such that $\|AP_{\alpha}-A_n\| < \frac{\epsilon}{4}$. Let $x \in \Mf$. Now as in the proof of Proposition \ref{prop5}, there is a net $(z_{\gamma})$ converging to $x$ such that the net $(U_{z_{\gamma}}A_nU_{z_{\gamma}})$ converges in weak operator topology. This limit will be denoted by $(A_x)_n$. It follows that $(A_x)_n$ is a band operator of band width at most $\omega$ and $\norm{A_x - (A_x)_n} < \frac{\epsilon}{4}$. Let $B \in \set{A_n} \cup \set{(A_x)_n : x \in \Mf \setminus \B^n}$ and choose $f \in L^p_{\nu}$ with $\norm{f} = 1$ and $\supp f \subseteq F$ such that
\[\|BP_{\alpha}f\| \geq \|BP_{\alpha}|_F\| - \frac{\epsilon}{4}.\]
We can apply the same reasoning as in the proof of Proposition \ref{prop5} (just reverse the inequalities and use the reverse triangle inequality) to obtain
\[\left(\sum\limits_{j = 1}^{\infty} \norm{BP_{\alpha}M_{\varphi_{j,t}^{1/p}}f}^p\right)^{1/p} \geq \|BP_{\alpha}f\| - \frac{\epsilon}{4}\]
for sufficiently small $t$. This implies
\[(\|BP_{\alpha}|_F\| - \frac{\epsilon}{2})\left(\sum\limits_{j = 1}^{\infty} \norm{M_{\varphi_{j,t}^{1/p}}f}^p\right)^{1/p} \leq \left(\sum\limits_{j = 1}^{\infty} \norm{BP_{\alpha}M_{\varphi_{j,t}^{1/p}}f}^p\right)^{1/p}.\]
Thus there exists a $j \in \N$ such that
\[(\|BP_{\alpha}|_F\| - \frac{\epsilon}{2})\norm{M_{\varphi_{j,t}^{1/p}}f} \leq \norm{BP_{\alpha}M_{\varphi_{j,t}^{1/p}}f}\]
for sufficiently small $t$. Of course, this implies $\|BP_{\alpha}|_F\| - \epsilon \leq \vertiiis{BP_{\alpha}|_F}_t$. As $t$ does not depend on $f$ or $B$, the result follows as in the proof of Proposition \ref{prop5}.
\end{proof}

\begin{thm} \label{thm5}
Let $p \leq 2$, $\alpha = (\frac{2}{p} - 1)(n+1) + \frac{2\nu}{p}$ and $A \in \Tf_{p,\nu}$. Then
\[\frac{1}{\norm{P_{\alpha}}\norm{P_{\nu}}}\norm{A + \Kc(A^p_{\nu})} \leq \sup\limits_{x \in \Mf \setminus \B^n} \norm{A_x} \leq \norm{P_{\nu}}\norm{A + \Kc(A^p_{\nu})}.\]
\end{thm}

\begin{proof}
Let $x \in \Mf \setminus \B^n$, $K \in \Kc(A^p_{\nu})$ and choose a net $(z_{\gamma})$ in $\B^n$ that converges to $x$. As $K$ is compact, we get $K_x = 0$ by \cite[Proposition 4.12, Theorem 5.5]{MiSuWi}. Banach-Steinhaus thus implies
\[\norm{A_x} = \norm{A_x+K_x} = \norm{(A+K)_x} \leq \sup\limits_{\gamma} \norm{U_{z_{\gamma}}^p(A+K)(U_{z_{\gamma}}^q|_{A^q_{\nu}})^*} \leq \norm{P_{\nu}}\norm{A+K},\]
where we used that $U_{z_{\gamma}}^p$ and $(U_{z_{\gamma}}^q)^*$ are isometries. Since this is true for all $K \in \Kc(A^p_{\nu})$ and $x \in \Mf \setminus \B^n$, we get
\[\sup\limits_{x \in \Mf \setminus \B^n} \norm{A_x} \leq \norm{P_{\nu}}\norm{A + \Kc(A^p_{\nu})}.\]

For the other inequality observe that
\[\norm{A + \Kc(A^p_{\nu})} \leq \inf\limits_{K \in \Kc(L^p_{\nu},A^p_{\nu})}\norm{AP_{\alpha} + K}.\]
Indeed, 
\[\norm{AP_{\alpha} + K} = \sup\limits_{\norm{f} = 1} \norm{(AP_{\alpha} + K)f} \geq \sup\limits_{\substack{f \in A^p_{\nu},\\ \norm{f} = 1}} \norm{(AP_{\alpha} + K)f} = \sup\limits_{\substack{f \in A^p_{\nu},\\ \norm{f} = 1}} \norm{(A + K)f} = \norm{A+K|_{A^p_{\nu}}}\]
for every $K \in \Kc(L^p_{\nu},A^p_{\nu})$. We will now show
\begin{equation} \label{eq_proof_thm5}
\inf\limits_{K \in \Kc(L^p_{\nu},A^p_{\nu})}\norm{AP_{\alpha} + K} \leq \sup\limits_{x \in \Mf \setminus \B^n} \norm{A_xT_{b_x}P_{\alpha}},
\end{equation}
which obviously implies the desired inequality. So assume that \eqref{eq_proof_thm5} is violated, i.e.~that there is an $\epsilon > 0$ such that
\[\inf\limits_{K \in \Kc(L^p_{\nu},A^p_{\nu})}\norm{AP_{\alpha} + K} > \sup\limits_{x \in \Mf \setminus \B^n} \norm{A_xT_{b_x}P_{\alpha}} + \epsilon.\]
In particular,
\[\norm{AP_{\alpha}|_{\B^n \setminus D(0,s)}} = \norm{AP_{\alpha}M_{1-\chi_{D(0,s)}}} = \norm{AP_{\alpha} - AP_{\alpha}M_{\chi_{D(0,s)}}} > \sup\limits_{x \in \Mf \setminus \B^n} \norm{A_xT_{b_x}P_{\alpha}} + \epsilon\]
for all $s > 0$ since $P_{\alpha}M_{\chi_{D(0,s)}} \in \Kc(L^p_{\nu},A^p_{\nu})$ by Proposition \ref{prop2.5}. Now, by Proposition \ref{prop9}, there is a $t \in (0,1)$ such that for all $s > 0$ we have
\[\vertiii{AP_{\alpha}|_{\B^n \setminus D(0,s)}}_t \geq \norm{AP_{\alpha}|_{\B^n \setminus D(0,s)}} - \frac{\epsilon}{2} > \sup\limits_{x \in \Mf \setminus \B^n} \norm{A_xT_{b_x}P_{\alpha}} + \frac{\epsilon}{2}.\]
In particular, for every $s > 0$ we get a $w_s \in \B^n$ such that
\[\norm{AP_{\alpha}M_{\chi_{D(w_s,r_t)}}} \geq \norm{AP_{\alpha}M_{\chi_{D(w_s,r_t) \setminus D(0,s)}}} > \sup\limits_{x \in \Mf \setminus \B^n} \norm{A_xT_{b_x}P_{\alpha}} + \frac{\epsilon}{2}.\]
It is clear that $w_s \to \partial\B^n$ as $s \to \infty$. Since $M_{\chi_{D(w_s,r_t)}} = U_{w_s}^pM_{\chi_{D(0,r_t)}}U_{w_s}^p$ and $P_{\alpha}U_{w_s}^p = U_{w_s}^pP_{\alpha}$ (see the proof of Proposition \ref{prop4}), we get
\[\norm{U_{w_s}^pAU_{w_s}^pP_{\alpha}M_{\chi_{D(0,r_t)}}} > \sup\limits_{x \in \Mf \setminus \B^n} \norm{A_xT_{b_x}P_{\alpha}} + \frac{\epsilon}{2}.\]
As $\Mf$ is compact, $(w_s)$ has a convergent subnet, denoted again by $(w_s)$, converging to some $x \in \Mf \setminus \B^n$. Thus $U_{w_s}^pAU_{w_s}^p|_{A^p_{\nu}}$ converges strongly to $A_xT_{b_x}$ and hence
\[\norm{U_{w_s}^pAU_{w_s}^pP_{\alpha}M_{\chi_{D(0,r_t)}}} \to \norm{A_xT_{b_x}P_{\alpha}M_{\chi_{D(0,r_t)}}}\]
since $P_{\alpha}M_{\chi_{D(0,r_t)}}$ is compact. This yields
\[\norm{A_xT_{b_x}P_{\alpha}M_{\chi_{D(0,r_t)}}} \geq \sup\limits_{x \in \Mf \setminus \B^n} \norm{A_xT_{b_x}P_{\alpha}} + \frac{\epsilon}{2},\]
which is certainly a contradiction. Thus $\inf\limits_{K \in \Kc(L^p_{\nu},A^p_{\nu})}\norm{AP_{\alpha} + K} \leq \sup\limits_{x \in \Mf \setminus \B^n} \norm{A_xT_{b_x}P_{\alpha}}$.
\end{proof}

For $p \geq 2$ we get the following corollary by using the adjoint of $P_{\alpha}$ instead.

\begin{cor} \label{cor5}
Let $p \geq 2$, $\alpha = (1 - \frac{2}{p})(n+1) + 2\nu(1-\frac{1}{p})$ and $A \in \Tf_{p,\nu}$. Then
\[\frac{1}{\norm{P_{\alpha}^*}\norm{P_{\nu}}}\norm{A + \Kc(A^p_{\nu})} \leq \sup\limits_{x \in \Mf \setminus \B^n} \norm{A_x} \leq \norm{P_{\nu}}\norm{A + \Kc(A^p_{\nu})}.\]
\end{cor}

\begin{lem} \label{lem7}
Let $p \leq 2$, $\alpha = (\frac{2}{p} - 1)(n+1) + \frac{2\nu}{p}$ and $A \in \Tf_{p,\nu}$. Then there exists a $y \in \Mf \setminus \B^n$ with
\[\|A_yT_{b_y}P_{\alpha}\| = \sup\set{\|A_xT_{b_x}P_{\alpha}\| : x \in \Mf \setminus \B^n}.\]
\end{lem}

\begin{proof}
Replacing all $\nu$ by $\norm{\cdot}$ and $\nu_t$ by $\vertiii{\cdot}_t$ in the proof of Lemma \ref{lem6} and using Proposition \ref{prop9} instead of Proposition \ref{prop5} one easily obtains a proof of Lemma \ref{lem7} (see also \cite[Theorem 3.2]{HaLiSe}).
\end{proof}

Let us summarize these results in a final theorem. This may be seen as an analogue of Theorem \ref{main_result} and a slight improvement of \cite[Theorem 5.2]{MiSuWi}. Unfortunately this result is far less complete than in the case of the spectrum and thus leaves some questions open: Are $\norm{A + \Kc(A^p_{\nu})}$ and $\sup\limits_{x \in \Mf \setminus \B^n} \norm{A_x}$ equal also for $p \neq 2$? And is the supremum also a maximum in case $p \neq 2$?

\begin{thm} \label{thm7}
Let $A \in \Tf_{p,\nu}$ and $\alpha = \abs{\frac{2}{p} - 1}(\nu+n+1) + \nu$. Then
\[\frac{1}{\norm{P_{\alpha}}\norm{P_{\nu}}}\norm{A + \Kc(A^p_{\nu})} \leq \sup\limits_{x \in \Mf \setminus \B^n} \norm{A_x} \leq \norm{P_{\nu}}\norm{A + \Kc(A^p_{\nu})},\]
where the norm of $P_{\alpha}$ is taken on $L^{\min\set{p,q}}_{\nu}$ ($\frac{1}{p} + \frac{1}{q} = 1$ as usual). Moreover,
\[\norm{A + \Kc(A^p_{\nu})} = \max\limits_{x \in \Mf \setminus \B^n} \norm{A_x}\]
if $p = 2$.
\end{thm}

\begin{proof}
The first statement is a combination of Theorem \ref{thm5} and Corollary \ref{cor5}. In case $p = 2$ we have $\norm{P_{\alpha}} = \norm{P_{\nu}} = 1$ and therefore
\[\norm{A + \Kc(A^p_{\nu})} = \sup\limits_{x \in \Mf \setminus \B^n} \norm{A_x}\]
(see also \cite[Theorem 5.6]{MiSuWi}). Moreover, as the norms of $A_xT_{b_x}P_{\alpha} = A_xP_{\nu}$ and $A_x$ coincide in this case, the second statement follows from Lemma \ref{lem7}.
\end{proof}

\section{Application to symbols of vanishing oscillation} \label{VO}

In this section we apply Theorem \ref{main_result} to the case of functions of vanishing oscillation. Even though the results obtained in this section are not completely new, it is worth mentioning that they are special cases of Theorem \ref{main_result}.

For $z \in \B^n$ and a bounded continuous ($\BC$) function $f \from \B^n \to \C$ we define
\[\Osc_z(f) := \sup\set{\abs{f(z)-f(w)} : w \in \B^n, \beta(z,w) \leq 1}\]
and $\VO_{\partial}(\B^n) := \set{f \in \BC(\B^n) : \lim\limits_{|z| \to 1} \Osc_z(f) = 0}$. Note that
\[C(\overline{\B^n}) \subset \VO_{\partial}(\B^n) \subset \BUC(\B^n).\]
Applying Corollary \ref{cor3} to Toeplitz operators with symbol in $\VO_{\partial}(\B^n)$, we obtain the following result:

\begin{prop} \label{prop8}
Let $f \in \VO_{\partial}(\B^n)$. Then
\[\spec_{\ess}(T_f) = f(\partial\B^n),\]
where $f(\partial\B^n)$ denotes the set of limit points of $f$ as $z \to \partial\B^n$.
\end{prop}

In case $f$ is contained in $C(\overline{\B^n})$, $f(\partial\B^n)$ is just the image of $f|_{\partial\B^n}$ and we obtain the classical result mentioned in the introduction.

\begin{proof}
Let $(z_{\gamma})$ be a net in $\B^n$ converging to some $x \in \Mf \setminus \B^n$ and consider $f \in \VO_{\partial}(\B^n) \subset \BUC(\B^n)$. Then we have $(f \circ \phi_{z_{\gamma}})(0) = f(z_{\gamma}) \to x(f)$. Moreover, since $\beta(\phi_{z_{\gamma}}(0),\phi_{z_{\gamma}}(w)) = \beta(0,w)$, we get
\[\abs{(f \circ \phi_{z_{\gamma}})(0) - (f \circ \phi_{z_{\gamma}})(w)} \leq \Osc_{z_{\gamma}}(f) \to 0\]
if $\beta(0,w) \leq 1$. Thus $(f \circ \phi_{z_{\gamma}})(w) \to x(f)$ uniformly on $\set{w \in \B^n : \beta(0,w) \leq 1}$. By repeating this argument and using that $\abs{\phi_{z_{\gamma}}(w)} \to 1$ if $\abs{z_{\gamma}} \to 1$, this generalizes to arbitary compact subsets of $\B^n$. Therefore $f \circ \phi_{z_{\gamma}}$ converges uniformly on compact sets to the constant function $x(f)$. Using \cite[p.~222]{MiSuWi}, we obtain
\[(T_f)_x = \slim\limits_{z_{\gamma} \to x} U^p_{z_{\gamma}}T_f(U^q_{z_{\gamma}}|_{A^q_{\nu}})^* = \slim\limits_{z_{\gamma} \to x} T_{b_{z_{\gamma}}}^{-1}T_{(f \circ \phi_{z_{\gamma}})b_{z_{\gamma}}}T_{b_{z_{\gamma}}}^{-1} = T_{b_x}^{-1}T_{x(f) \cdot b_x}T_{b_x}^{-1} = x(f)T_{b_x}^{-1}.\]
As $T_{b_x}^{-1}$ is always invertible, this implies that $(T_f)_x$ is invertible if and only if $x(f) \neq 0$. Corollary \ref{cor3} thus simplifies to
\begin{equation} \label{essspec_VO}
\spec_{\ess}(T_f) = \bigcup\limits_{x \in \Mf \setminus \B^n} x(f) = f(\partial\B^n)
\end{equation}
in this case.
\end{proof}

Let us add two final remarks to this result.

\begin{rem}
If we introduce $\Mf_{\VO_{\partial}}$ as the maximal ideal space of $\VO_{\partial}$, we can formulate Proposition \ref{prop8} like this:
\begin{equation} \label{essspec_VO2}
\spec_{\ess}(T_f) = \bigcup\limits_{x \in \Mf_{\VO_{\partial}} \setminus \B^n} x(f) =: f(\Mf_{\VO_{\partial}} \setminus \B^n)
\end{equation}
for $f \in \VO_{\partial}$. This can be seen as follows. Let $\iota \from \VO_{\partial} \to \BUC(\B^n)$ be the inclusion mapping. By transposition, this induces the continuous map $\pi \from \Mf \to \Mf_{\VO_{\partial}}$, $(\pi(x))(f) := x(\iota(f))$. \eqref{essspec_VO2} thus follows from \eqref{essspec_VO}.
\end{rem}

\begin{rem}
In \cite{PeVi,TaVi,Zhu87} similar results were shown for symbols in $\VMO_{\partial}(\B^n) \cap L^{\infty}(\B^n)$ (see \cite{BBCZ} or \cite{Zhu87} for definitions and descriptions). We can also recover this result from Proposition \ref{prop8}: If $f \in \VMO_{\partial}(\B^n) \cap L^{\infty}(\B^n)$, then
\[\spec_{\ess}(T_f) = \tilde{f}(\partial\B^n),\]
where $\tilde{f}$ denotes the Berezin transform of $f$. Indeed, by \cite[Theorem B]{BBCZ}, $\tilde{f}$ is contained in $\VO_{\partial}(\B^n)$ and $T_{f - \tilde{f}}$ is compact (see also \cite[Theorem 5.5]{MiSuWi}). Thus the assertion follows.
\end{rem}

\section*{Acknowledgements}

I am very grateful to Wolfram Bauer, Marko Lindner and Christian Seifert for their support and useful discussions. I would also like to thank Robert Fulsche for careful reading and pointing out a lot of typos.

\bigskip

\noindent
Raffael Hagger\\
Institute of Analysis\\
Leibniz Universit\"at Hannover\\
Welfengarten 1\\
30167 Hannover\\
GERMANY\\
raffael.hagger@math.uni-hannover.de


\begin{thebibliography}{}
\bibitem{ACM} S.~Axler, J.~Conway, G.~McDonald, \emph{Toeplitz Operators on Bergman Spaces}, Can.~J.~Math., Vol.~34, No.~2 (1982), 466-483.
\bibitem{BaCo} W.~Bauer, L.~Coburn, \emph{Toeplitz Operators with Uniformly Continuous Symbols}, Integr.~Equ. Oper.~Theory 83 (2015), 25-34.
\bibitem{BBCZ} D.~B\'{e}koll\'{e}, C.~Berger, L.~Coburn, K.~Zhu, \emph{BMO in the Bergman metric on bounded symmetric domains}, J.~Funct.~Anal., Vol.~93 (1990), 310-350.
\bibitem{BeTe} D.~B\'{e}koll\'{e}, A.~Temgoua Kagou, \emph{Reproducing Kernels and $L^p$-Estimates for Bergman Projections in Siegel Domains of Type II}, Studia Math. 115 (1995), 219-239.
\bibitem{BeDra} G.~Bell, A.~Dranishnikov, \emph{Asymptotic Dimension}, Topology and its Applications, Vol.~155, No.~12 (2018), 1265-1296.
\bibitem{BeCo} C.~Berger, L.~Coburn, \emph{Toeplitz Operators on the Segal-Bargmann Space}, T.~Am.~Math.~Soc., Vol.~301, No.~2 (1987), 813-129.
\bibitem{CaGo} G.~Carsson, B.~Goldfarb, \emph{On Homological Coherence of Discrete Groups}, J.~Algebra, Vol.~276, No.~2 (2004), 502-514.
\bibitem{ChoLe} B.~Choe, Y.~Lee, \emph{The Essential Spectra of Toeplitz Operators with Symbols in $H^{\infty} + C$}, Math.~Japon., Vol.~45, No.~1 (1997), 57-60.
\bibitem{Coburn} L.~Coburn, \emph{Singular Integral Operators and Toeplitz Operators on Odd Spheres}, Indiana Univ.~Math.~J.~23 (1973), 433-439.
\bibitem{Englis} M.~Engli\v{s}, \emph{Compact Toeplitz Operators via the Berezin Transform on Bounded Symmetric Domains}, Integr.~Equ.~Oper.~Theory 33 (4), 426-455 (1999).
\bibitem{Englis2} M.~Engli\v{s}, \emph{A Mean Value Theorem on Bounded Symmetric Domains}, P.~Am.~Math.~Soc., Vol.~127, No.~11, 3259-3268 (1999).
\bibitem{FaKo} J.~Faraut, A.~Koranyi, \emph{Function Spaces and Reproducing Kernels on Bounded Symmetric Domains}, J.~Funct.~Anal. 88 (1990), 64-89.
\bibitem{HaLiSe} R.~Hagger, M.~Lindner, M.~Seidel: \emph{Essential Pseudospectra and Essential Norms of Band-Dominated Operators}, J.~Math.~Anal.~Appl., Vol.~437, No.~1 (2016), 255-291.
\bibitem{Lindner} M.~Lindner, \emph{Infinite Matrices and their Finite Sections}, Birkh\"auser Verlag, Basel, Boston, Berlin, 2006.
\bibitem{LiSe} M.~Lindner, M.~Seidel, \emph{An Affirmative Answer to a Core Issue on Limit Operators}, J.~Funct.~Anal., Vol.~267, No.~3 (2014), 901-917.
\bibitem{McDonald} G.~McDonald, \emph{Fredholm Properties of a Class of Toeplitz Operators on the Ball}, Indiana Univ.~Math.~J., Vol.~26, No.~3 (1977), 567-576.
\bibitem{MiSuWi} M.~Mitkovski, D.~Su\'{a}rez, B.~Wick, \emph{The Essential Norm of Operators on $A_{\alpha}^p(\B_n)$}, Integr.~Equ.~Oper.~Theory 75 (2013), 197-233.
\bibitem{PeVi} A.~Per\"al\"a, A.~Virtanen, \emph{A Note on the Fredholm Properties of Toeplitz Operators on Weighted Bergman Spaces with Matrix-Valued Symbols}, Oper.~Matrices, Vol.~5, No.~1 (2011), 97-106.
\bibitem{RaRoSi} V.~Rabinovich, S.~Roch, B.~Silbermann, \emph{Fredholm Theory and Finite Section Method for Band-Dominated Operators}, Integr.~ Equat.~Oper.~Th., Vol.~30, No.~4 (1998), 452-495.
\bibitem{RaRoSi2} V.~Rabinovich, S.~Roch, B.~Silbermann, \emph{Limit Operators and Their Applications in Operator Theory}, Birkh\"auser Verlag, Basel, Boston, Berlin, 2004.
\bibitem{Seidel} M.~Seidel, \emph{Fredholm Theory for Band-Dominated and Related Operators: a Survey}, Linear Algebra Appl., Vol.~445 (2014), 373-394.
\bibitem{SpaWi} J.~\v{S}pakula, R.~Willett, \emph{A Metric Approach to Limit Operators}, Trans.~Amer.~Math.~Soc., Vol.~369 (2017), 263-308.
\bibitem{StroeZhe} K.~Stroethoff, D.~Zheng, \emph{Toeplitz and Hankel Operators on Bergman Spaces}, Trans.~Amer. Math.~Soc., Vol.~329, No.~2 (1992), 773-794.
\bibitem{Suarez} D.~Su\'{a}rez, \emph{The Essential Norm of Operators in the Toeplitz Algebra on $A^p(\B_n)$}, Indiana Univ.~Math.~J., Vol.~56, No.~5 (2007), 2185-2232.
\bibitem{TaVi} J.~Taskinen, J.~Virtanen, \emph{Toeplitz Operators on Bergman Spaces with Locally Integrable Symbols}, Rev.~Mat.~Iberoam., Vol.~26, No.~2 (2010), 693-706.
\bibitem{Upmeier} H.~Upmeier, \emph{Toeplitz Operators and Index Theory in Several Complex Variables}, Operator Theory: Advances and Applications, Vol.~81, Birk\"auser Verlag, Basel, 1996.
\bibitem{Zaanen} A.~Zaanen, \emph{Introduction to Operator Theory in Riesz Spaces}, Springer-Verlag, New York and Heidelberg, 1997.
\bibitem{Zeng} X.~Zeng, \emph{Toeplitz Operators on Bergman Spaces}, Houston J.~Math., Vol.~18, No.~3 (1992), 387-407.
\bibitem{Zhu87} K.~Zhu, \emph{VMO, ESV, and Toeplitz Operators on the Bergman Space}, Trans.~Amer.~Math.~Soc., Vol.~302, No.~2 (1987), 617-646.
\bibitem{Zhu} K.~Zhu, \emph{Spaces of Holomorphic Functions in the Unit Ball}, Graduate Texts in Mathematics, Vol.~226, Springer-Verlag, New York, 2005.
\end{thebibliography}
\end{document}